\pdfoutput=1
\documentclass[12pt,a4paper]{amsart}
\RequirePackage[l2tabu, orthodox]{nag}

\usepackage[T1]{fontenc}
\usepackage{lmodern}
\usepackage[kerning, spacing, tracking]{microtype}
\usepackage[utf8]{inputenc}

\usepackage[style=numams, 
            sorting=nyt, 
            isbn=false,  
            backref=true,
            backend=biber]{biblatex}
\usepackage{amsmath,amssymb, amsthm}
\usepackage{mathtools}
\usepackage{leftidx}
\usepackage{enumitem}

\usepackage{tikz-cd}
\usepackage{stmaryrd}  

\usepackage{tikz}
\usetikzlibrary{positioning}

\usepackage[colorlinks, linkcolor=blue, citecolor=blue]{hyperref}

\usepackage[hcentering, hscale=0.7, vscale=0.74 ]{geometry}

\newtheorem{introthm}{Theorem}

\newtheorem{introcor}[introthm]{Corollary}
\newtheorem{introprop}[introthm]{Proposition}

\swapnumbers
\newtheorem{thm}{Theorem}[section]
\newtheorem{cor}[thm]{Corollary}
\newtheorem{lemma}[thm]{Lemma}
\newtheorem{prop}[thm]{Proposition}

\theoremstyle{definition}
\newtheorem{defi}[thm]{Definition}

\newtheorem{hyp}[thm]{Hypothesis}

\theoremstyle{remark}
\newtheorem{remark}[thm]{Remark}

\newlist{enums}{enumerate}{2}  
\setlist[enums,1]{label=\textup{(\alph*)}}
\setlist[enums,2]{label=\textup{(\roman*)}}

\renewcommand{\phi}{\varphi}
\renewcommand{\theta}{\vartheta}
\newcommand{\eps}{\varepsilon}

\renewcommand{\geq}{\geqslant}
\renewcommand{\leq}{\leqslant}

\newcommand{\nbd}{\nobreakdash-\hspace{0pt}}  
\newcommand{\defemph}[1]{\textbf{#1}} 

\newcommand{\ints}{\mathbb{Z}}
\newcommand{\rats}{\mathbb{Q}}
\newcommand{\compl}{\mathbb{C}}

\newcommand{\crp}[1]{\mathbb{#1}}     

\DeclarePairedDelimiterX{\brcls}[3]{\llbracket}{\rrbracket}{#1,#2,#3}

\DeclarePairedDelimiterX{\cohcl}[2]{\lbrack}{\rbrack}{#1,#2}  

\newcommand{\iso}{\cong}    
\newcommand{\nteq}{\trianglelefteq} 
\newcommand{\nt}{\triangleleft}
\newcommand{\cconj}[1]{\overline{#1}}  
\newcommand{\into}{\hookrightarrow}  
\newcommand{\tensor}{\otimes}
\newcommand{\oG}{\widehat{G}}
\newcommand{\oH}{\widehat{H}}

\newcommand{\macl}{\operatorname{\mathcal{M}}}
\newcommand{\cyccl}{\operatorname{\mathcal{C}}}
\DeclarePairedDelimiter{\abs}{\lvert}{\rvert}
\DeclarePairedDelimiter{\erz}{\langle}{\rangle}

\DeclarePairedDelimiterX{\ipch}[2]{[}{]}{#1,#2}  

\DeclareMathOperator{\Ker}{Ker}        
\DeclareMathOperator{\Aut}{Aut}
\DeclareMathOperator{\Z}{\mathbf{Z}}
\DeclareMathOperator{\C}{\mathbf{C}}         
\DeclareMathOperator{\Gal}{Gal}  
\DeclareMathOperator{\Irr}{Irr}
\DeclareMathOperator{\Lin}{Lin}
\DeclareMathOperator{\mat}{\mathbf{M}}      
\DeclareMathOperator{\enmo}{End} 
\DeclareMathOperator{\Tr}{Tr}
\DeclareMathOperator{\SC}{\mathcal{SC}}
\DeclareMathOperator{\BrCliff}{BrCliff}
\DeclareMathOperator{\Cores}{Cores}
\DeclareMathOperator{\Res}{Res}
\DeclareMathOperator{\Inf}{Inf}

\DeclareMathOperator{\Br}{Br}  
\hypersetup{pdfauthor={Frieder Ladisch},
            pdftitle={On Clifford theory with Galois action},
            pdfkeywords={Brauer-Clifford group, Clifford theory,
                         Character theory of finite groups,
                         Galois theory, Schur indices}
            }

\begin{document}
\title{On Clifford theory with Galois action}
\author{Frieder Ladisch}
\thanks{Author partially supported by the DFG (Project: SCHU 1503/6-1)}
\address{Universität Rostock,
         Institut für Mathematik,
         Ulmenstr.~69, Haus~3,
         18057 Rostock,
         Germany}
\email{frieder.ladisch@uni-rostock.de}
\begin{abstract}
  Let $\widehat{G}$ be a finite group, 
  $N $ a normal subgroup of $\widehat{G}$ 
  and $\theta\in \operatorname{Irr}N$. 
  Let $\mathbb{F}$ be a subfield of the complex numbers
  and assume that the Galois orbit of $\theta$ over $\mathbb{F}$
  is invariant in $\widehat{G}$.
  We show that there is another triple
  $(\widehat{G}_1,N_1,\theta_1)$ of the same form,
  such that the character theories of $\widehat{G}$ over $\theta$
  and of $\widehat{G}_1$ over $\theta_1$ are essentially ``the same''
  \emph{over the field $\mathbb{F}$}
  and such that the following holds:
  $\widehat{G}_1$ has a cyclic normal subgroup $C$ contained in $N_1$,
  such that
  $\theta_1=\lambda^{N_1}$ for some 
  linear character $\lambda$ of $C$,
  and such that $N_1/C$ is isomorphic to the (abelian) Galois group
  of the field extension $\mathbb{F}(\lambda)/\mathbb{F}(\theta_1)$.
  More precisely, having ``the same'' character theory means 
  that both triples yield the same
  element of the Brauer-Clifford group
  $\operatorname{BrCliff}(G,\mathbb{F}(\theta))$
  defined by A.~Turull.
\end{abstract}
\maketitle

\section{Introduction}
\subsection{Motivation}
Clifford theory is concerned with the characters
of a finite group lying over one fixed character of a normal subgroup.
So let $\oG$ be a finite group and
$N\nteq \oG$ a normal subgroup.
Let $\theta\in \Irr N$, 
where $\Irr N$ denotes the set of irreducible complex valued 
characters of the group $N$, as usual.
We write $\Irr(\oG\mid\theta)$ for the set of irreducible characters
of $\oG$ which lie above $\theta$ in the sense that their restriction
to $N$ has $\theta$ as constituent.

In studying $\Irr(\oG\mid \theta)$,
it is usually no loss of generality to assume that 
$\theta$ is invariant in $\oG$, using the well known 
\emph{Clifford correspondence}~\cite[Theorem~6.11]{isaCTdov}.
In this situation, $(\oG, N, \theta)$ is often called a character triple.
Then a well known theorem tells us that 
there is an ``isomorphic'' character triple $(\oG_1,N_1,\theta_1)$ 
such that $N_1 \subseteq \Z(\oG_1)$~\cite[Theorem~11.28]{isaCTdov}.
Questions about $\Irr(\oG\mid \theta)$ can often be reduced
to questions about $\Irr(\oG_1\mid \theta_1)$,
which are usually easier to handle.
This result is extremely useful, for example in reducing
questions about characters of finite groups to 
questions about characters of finite simple or quasisimple groups.

Some of these questions involve Galois automorphisms 
or even Schur indices~\cite{nav04,turull08b}
(over some fixed field $\crp{F}\subseteq \compl$, say).
Unfortunately, both of the above reductions are 
not well behaved with respect to Galois action on characters
and other rationality questions (like Schur indices of the involved
characters).
The first reduction (Clifford correspondence)
can be replaced by a reduction
to the case where $\theta$ is 
semi-invariant over the given field 
$\crp{F}$~\cite[Theorem~1]{rieschm96}.
(This means that the Galois orbit of $\theta$ is invariant
 in the group $\oG$.)

Now assuming that the character triple
$(\oG,N,\theta)$ is such that $\theta$ is semi-invariant in $\oG$
over some field $\crp{F}$,
usually we can not find a character triple
$(\oG_1,N_1,\theta_1)$ with $N_1\subseteq \Z(\oG_1)$,
and such that these character triples are 
``isomorphic over the field $\crp{F}$''.
We will give an exact definition of 
``isomorphic over $\crp{F}$'' below, using machinery developed
by Alexandre Turull~\cite{turull09,turull11}.
For the moment, it suffices to say that a correct definition
should imply that $\oG/N\iso \oG_1/N_1$ and that 
there is a bijection
between $\bigcup_{\alpha}\Irr(\oG\mid \theta^{\alpha})$ and 
$\bigcup_{\alpha}\Irr(\oG_1 \mid \theta_1^{\alpha})$
(unions over a Galois group)
commuting with field automorphisms over $\crp{F}$
and preserving Schur indices over $\crp{F}$.
Now if, for example, $\rats(\theta) = \rats(\sqrt{5})$ (say),
then it is clear that we can not 
have $\rats(\theta)=\rats(\theta_1)$ with $\theta_1$ linear.
The main result of this paper, as described in the abstract,
provides a possible substitute: 
At least we can find an ``isomorphic'' character triple 
$(\oG_1, N_1, \theta_1)$, where $N_1$ is cyclic by abelian
and $\theta_1$ is induced from a cyclic normal subgroup
of $\oG_1$.
This result is probably the best one can hope for,
if one wants to take into account Galois action and Schur indices.

\subsection{Notation}
To state the main result precisely, we need some notation.
Instead of character triples, we find it more convenient to use
Clifford pairs as introduced in~\cite{ladisch15b}.
We recall the definition.
Let $\oG$ and $G$ be finite groups
and let $\kappa\colon \oG \to G$ be
a surjective group homomorphism with kernel
$\Ker \kappa = N$. 
Thus 
\[
    \begin{tikzcd}
       1 \rar & N \rar & \oG \rar{\kappa} & G \rar & 1      
    \end{tikzcd}
\]
is an exact sequence, and $\oG/N\iso G$ via $\kappa$.
We say that
$(\theta,\kappa)$ is a \defemph{Clifford pair}
over $G$.
(Note that $\oG$, $G$ and $N$ are determined by $\kappa$ as
the domain, the image and the kernel of $\kappa$, respectively.)
We usually want to compare different Clifford pairs over the 
same group $G$, but with different groups $\oG$ and $N$.

Let $\crp{F}\subseteq \compl$ be a field.
(For simplicity of notation, we work with subfields of $\compl$, 
 the complex numbers,
 but of course one can replace $\compl$ by any 
 algebraically closed field of characteristic $0$ and assume that 
 all characters take values in this field.)
Then
$\theta\in \Irr N$ is called 
\defemph{semi-invariant} in $\oG$ over $\crp{F}$
(where $N\nteq \oG$),
if for every
$g\in \oG$, there is a field automorphism
$\alpha=\alpha_g\in \Gal(\crp{F}(\theta)/\crp{F})$
such that $\theta^{g\alpha}=\theta$.
In this situation,
the map $g\mapsto \alpha_g$ 
actually defines an action of $G$ on the field
$\crp{F}(\theta)$~\cite[Lemma~2.1]{i81b}.

To handle Clifford theory over small fields, 
Turull~\cite{turull09,turull11}
has introduced the \defemph{Brauer-Clifford group}.
For the moment, it is enough to know that the 
Brauer-Clifford group is a certain set
$\BrCliff(G,\crp{E})$ for any group $G$ and any field~$\crp{E}$ 
on which $G$ acts.
Given a Clifford pair 
$(\theta,\kappa)$ and a field $\crp{F}$ such that
$\theta$ is semi-invariant over $\crp{F}$,
the group $G$ acts on $\crp{F}(\theta)$
and the Brauer-Clifford group
$\BrCliff(G,\crp{F}(\theta))$ is defined.
Turull~\cite[Definition~7.7]{turull09} shows how to associate a 
certain element
$\brcls{\theta}{\kappa}{\crp{F}}\in \BrCliff(G,\crp{F}(\theta))$
with $(\theta,\kappa)$ and $\crp{F}$.
Moreover, if $(\theta,\kappa)$
and $(\theta_1,\kappa_1)$ are two pairs over $G$ such that
$\theta$ and $\theta_1$ are semi-invariant over 
$\crp{F}$ and induce the same action of $G$ on
$\crp{F}(\theta)=\crp{F}(\theta_1)$,
and if 
$\brcls{\theta}{\kappa}{\crp{F}} =
   \brcls{\theta_1}{\kappa_1}{\crp{F}}$,
then the character theories of $\oG$ over $\theta$
and of $\oG_1$ over $\theta_1$ are 
essentially ``the same'', 
including rationality properties
over the field $\crp{F}$.
(See~\cite[Theorem~7.12]{turull09} for the exact statement.) 
This justifies it to view such Clifford pairs 
as ``isomorphic over $\crp{F}$''.

\subsection{Main result}
The following is the main result of this paper.
For simplicity, we state it for subfields of 
the complex numbers~$\compl$,
but in fact $\compl$ can stand for any algebraically
closed field of characteristic $0$,
if all characters are assumed to take values in that fixed field
$\compl$.
(As usual, $\Lin C$ denotes the set of \emph{linear} characters 
 of a group $C$.)
\begin{introthm}\label{main}
  Let\/ $\crp{F}\subseteq \compl$ be a field, let
  \[
    \begin{tikzcd}
       1 \rar & N \rar[hook] & \oG \rar{\kappa} & G \rar & 1    
    \end{tikzcd}
  \]
  be an exact sequence of finite groups
  and let $\theta\in \Irr N$ be semi-invariant 
  in $\oG$ over\/ $\crp{F}$.
  Then there is another exact sequence of finite groups
  \[
      \begin{tikzcd}
         1 \rar & N_1 \rar[hook] & \oG_1 \rar{\kappa_1} & G \rar & 1 
      \end{tikzcd}
  \]
  and $\theta_1\in \Irr N_1$,
  such that\/ $\crp{F}(\theta)=\crp{F}(\theta_1)$ as $G$-fields, 
  such that\/
  \[\brcls{\theta}{\kappa}{\crp{F}} =
    \brcls{\theta_1}{\kappa_1}{\crp{F}}
    \quad\text{in} \quad 
    \BrCliff(G, \crp{F}(\theta))
    ,
  \] 
  and such that the following hold:
  \begin{enums}
  \item \label{i:maincns}
        $\oG_1$ has a cyclic normal subgroup
        $C\nteq \oG_1$ with $C\leq N_1$,
  \item \label{i:mainind}
        there is a faithful $\lambda\in \Lin C$ 
        with $\theta_1=\lambda^{N_1}$,
  \item \label{i:mainsemiinv}
        $\lambda$ is semi-invariant in $\oG_1$ over\/ 
        $\crp{F}$
        and 
  \item \label{i:maingaliso}
        $N_1/C\iso \Gal(\crp{F}(\lambda)/\crp{F}(\theta))$.
  \end{enums}
\end{introthm}
As mentioned before, 
Turull's result~\cite[Theorem~7.12]{turull09} yields that
in the situation of Theorem~\ref{main},
there are correspondences of characters with 
good compatibility properties. 
We mention a few properties here, and refer the reader to 
Turull's paper for a more complete list:
\begin{introcor}
  In the situation of Theorem~\ref{main}, 
  for each subgroup $H\leq G$ there is a bijection
  between 
  \[ \ints[\Irr(\kappa^{-1}(H) \mid \theta)]
     \quad \text{and} \quad
     \ints[ \Irr( \kappa_1^{-1}(H) \mid \theta_1) ].
  \]
  The bijections can be chosen such that their union commutes
  with restriction and induction of characters,
  with field automorphisms over the field\/ $\crp{F}$,
  and with multiplications of characters of $H$,
  and such that it preserves the inner product of class functions
  and fields of values and Schur indices
  (even elements in the Brauer group) over $\crp{F}$.
\end{introcor}

We can say something more about the group $\oG_1$ in the main theorem.
\begin{introprop}\label{p:subgroupsmain}
  In the situation of Theorem~\ref{main} and 
  for $\widehat{U}_1 = (\oG_1)_{\theta_1}$ and $V= (\oG_1)_{\lambda}$
  (the inertia groups of $\theta_1$ and $\lambda$ in $\oG_1$),
  we have $\widehat{U}_1 = VN_1$ and $N_1 \cap V = C$
  (cf.~Figure~\ref{fig:mainthm}), and
  \[ \brcls{\theta_1}{(\kappa_1)_{|\widehat{U}_1}}{\crp{F}(\lambda)}
     =
     \brcls{\lambda}{(\kappa_1)_{|V}}{\crp{F}(\lambda)}.
  \]
  For every subgroup $X$ with
  $N_1\leq X \leq \widehat{U}_1$, induction yields a bijection
  \[ \Irr(X\cap V\mid \lambda)
     \ni \psi  \mapsto \psi^X
     \in \Irr( X \mid \theta)
  \]
  commuting with field automorphisms over
  $\crp{F}(\lambda)$ and preserving
  Schur indices over $\crp{F}(\lambda)$.
\end{introprop}
\begin{figure}[ht]
  \begin{tikzpicture}[on grid=true, inner sep=2pt]
      \node (U) {$\widehat{U}_1$};
      \node (G) [above right = 1.5 of U] 
                {$\oG_1$};
      \node (N) [below left=2.3 of U,
                 label=185:{$\lambda^{N_1}=\theta_1$}] 
                {$N_1$};
      \node (V) [below right=1.7 of U] 
                {$V$};
      \node (C) [below right=1.7 of N,
                 label={-25:{$\lambda$}} ] 
                {$C$};  
      \draw (U)--(G) node[auto=left, midway] {$\Gamma$};
      \draw (C)--(N) node[auto=left, midway] {$\Delta$};
      \draw (V)--(U) 
            node[auto=right, midway] 
            {$\Delta=\Gal(\crp{F}(\lambda)/\crp{F}(\theta))$};
      \draw (N)--(U); 
      \draw (V)--(C); 
      \node [right=0.5 of G, anchor= north west] 
            {($\Gamma\leq \Gal(\crp{F}(\theta)/\crp{F})$)};
  \end{tikzpicture}
  \caption{Subgroups of $\oG_1$ in Theorem~\ref{main}}
  \label{fig:mainthm}
\end{figure}
Observe that since $\lambda$ is faithful, we actually have
$V=\C_{\oG_1}(C)$ and $C\subseteq \Z(V)$.
So over the bigger field $\crp{F}(\lambda)$ and in the smaller group
$\widehat{U}_1$,
we can replace the Clifford pair
$(\theta_1, (\kappa_1)_{|\widehat{U}_1})$
by the even simpler pair $(\lambda, (\kappa_1)_{|V})$.
This is, of course, just the classical result mentioned before,
which is usually proved using the theory of projective representations
and covering groups.

We will prove Proposition~\ref{p:subgroupsmain} at the end of 
Section~\ref{sec:subext} as a special case of other, auxiliary
results, which are also needed for the proof of Theorem~\ref{main}.
Everthing else in this paper is devoted to
the proof of the main result, Theorem~\ref{main},
which is proved through a series of reductions.

\subsection{Relation to earlier results}
A number of people have studied Clifford theory over small fields,
including Dade~\cite{dade74,dade81b,dade08}, Isaacs~\cite{i81b},
Schmid~\cite{schmid85,schmid88} and Riese~\cite{riese95},
cf.~\cite{rieschm96}.
While we use some ideas of these authors,
most important for our paper is the theory of the Brauer-Clifford group
as developed by Turull~\cite{turull09,turull09b,turull11},
which supersedes in some sense his earlier theory of 
Clifford classes~\cite{turull94}.
In particular, it is essential for our proof 
that the Brauer-Clifford group and certain subsets of it are indeed 
groups, a fact which, it seems to me, 
has not been important in the applications
of the Brauer-Clifford group~\cite{turull08c,turull13b} so far.

A paper of Dade~\cite{dade74} contains, between the lines, a
result similar to our main theorem (but in a more narrow situation):
Dade studies the situation where (in our notation)
$\theta$ is invariant in $\oG$ and has values in 
$\crp{F}$.
Then a cohomology class $[\theta]\in H^2(G,\crp{F}^*)$ is defined.
This class determines part of the character theory of $\oG$ over 
$\theta$, but not completely, since it does not take into account 
the Schur index of $\theta$ itself.
After proving some properties of this cohomology class, 
Dade shows that all cohomology classes with these properties occur,
by constructing examples. 
When examining this construction, one will find that the examples
have almost the same properties as the group $\oG_1$ in 
Theorem~\ref{main}.
One can deduce that all such cohomology classes come from 
character triples as in Theorem~\ref{main}.

Theorem~\ref{main} contains the classical
result that every simple direct summand of the group algebra
$\crp{E}N$ of a finite group $N$ over a field $\crp{E}$
of characteristic zero
is equivalent to a cyclotomic algebra~\cite{Yamada74}.
(This is the case $G=1$ of Theorem~\ref{main}.)
Our proof of Theorem~\ref{main} uses ideas from the proof
of this result, as presented by Yamada~\cite{Yamada74}.

\section{Review of the Brauer-Clifford group}
\label{sec:review}
Let us briefly recall the relevant definitions.
Details can be found in the papers of Turull~\cite{turull09,turull11},
see also~\cite{HermanMitra11}.
Let $G$ be a group.
A \defemph{$G$-ring} is a ring $Z$ (with~$1$) together with an action
of $G$ on $Z$ via (unital) ring automorphisms.
We use exponential notation $z\mapsto z^g$ to denote such an action.
Let $Z$ be a commutative $G$-ring. 
A \defemph{$G$\nbd algebra over $Z$} is a $G$ ring $A$ together
with a homomorphism of $G$\nbd rings
$\eps\colon Z\to \Z(A)$.
This means that $A$ is an algebra over $Z$ and that the algebra unit
$\eps\colon Z\to A$ has the property
$\eps(z^g)=\eps(z)^g$.
We usually suppress mention of the algebra unit $\eps$
and simply write $za$ for $\eps(z)a$.

We only need $G$\nbd algebras over fields
in this paper.
Let $Z$ be a field on which $G$ acts.
A $G$\nbd algebra over $Z$ is called \defemph{central simple},
if it is central simple as algebra over $Z$.
The \defemph{Brauer-Clifford group} $\BrCliff(G,Z)$
is the set of equivalence classes of central simple 
$G$-algebras over $Z$
under a certain equivalence relation.
To define this equivalence relation, we need 
the \defemph{skew group ring}  $ZG$ of $G$ over $Z$
(also called the \defemph{crossed product}),
which is the set of formal sums
$\sum_{g\in G} gc_g$, ($c_g\in Z$)
with multiplication defined by
\[ \Big(\sum_{g\in G} gc_g \Big)
   \Big(\sum_{h\in G} h d_h\Big) 
   = \sum_{g,h} gh c_g^h d_h.
\]
If $V$ is a right $ZG$\nbd module, then
$\enmo_{Z}V$ is a $G$\nbd algebra over $Z$,
called a \defemph{trivial $G$\nbd algebra}.
Two $G$\nbd algebras $S$ and $T$ over $Z$ are called 
\defemph{equivalent}, if there are $ZG$\nbd modules
$V$ and $W$ such that
\[ S \tensor_Z \enmo_Z V \iso T\tensor_Z \enmo_Z W
\]
as $G$\nbd algebras over $Z$.
The Brauer-Clifford group $\BrCliff(G,Z)$ 
is the set of equivalence classes
of central simple $G$\nbd algebras over $Z$, with multiplication
induced by tensoring over $Z$.

The Brauer-Clifford group is an abelian torsion 
group~\cites[Theorem~3.10]{turull09}[Theorem~5]{HermanMitra11}.

Next let $(\theta,\kappa\colon \oG\to G)$ be 
a Clifford pair and $\crp{F}\subseteq \compl$ a field.
Assume that $\theta$ is semi-invariant in $\oG$ over $\crp{F}$.
(The last assumption is not necessary in Turull's theory,
 but we only need this case, and the notation 
 can be simplified somewhat in this case.)
So for every
$g\in \oG$, there is a field automorphism
$\alpha_g\in \Gal(\crp{F}(\theta)/\crp{F})$
such that $\theta^{g\alpha_g}=\theta$.
The map $g\mapsto \alpha_g$ 
defines an action of $G$ on the field
$\crp{F}(\theta)$~\cite[Lemma~2.1]{i81b}
and thus the 
Brauer-Clifford group $\BrCliff(G, \crp{F}(\theta))$
is defined.
Turull showed~\cite[Definition~7.7]{turull09} 
how to associate an element
\[\brcls{\theta}{\kappa}{\crp{F}}\in \BrCliff(G,\crp{F}(\theta))
\]
to $(\theta,\kappa,\crp{F})$.
We recall the construction.
Let 
$e=e_{(\theta,\crp{F})}\in \Z(\crp{F}N)$
be the central primitive idempotent of $ \crp{F}N $ 
corresponding to $\theta$.
Note that $e$ is invariant in $\oG$ since $\theta$ is semi-invariant.
An $\crp{F}\oG$\nbd module $V$ is called
\defemph{$\theta$\nbd quasihomogeneous} if $Ve=V$.
Let $V$ be a nonzero quasihomogeneous $\crp{F}\oG$-module.
Then $S=\enmo_{\crp{F}N}V$ is a $G$\nbd algebra.
For $z\in\Z(\crp{F}Ne)$, the map $v\mapsto vz$
is in $S$, and this defines a canonical isomorphism
$\Z(\crp{F}Ne)\iso \Z(S)$ of $G$\nbd algebras.
The central character $\omega_{\theta}$ belonging
to $\theta$ defines an isomorphism
$\Z(\crp{F}N e)\iso \crp{F}(\theta)$, 
which commutes with the action of $G$
(namely, we have 
$\omega_{\theta}(z^g)= \omega_{\theta}(z)^{\alpha_g}$).
Thus we can view $S$ as a central simple $G$\nbd algebra over 
$\crp{F}(\theta)$.
The equivalence class of $S$
 does not depend on the choice 
of the quasihomogeneous module $V$~\cite[Theorem~7.6]{turull09}.
Following Turull, we write
$\brcls{\theta}{\kappa}{\crp{F}}$
for this equivalence class.
(Actually, Turull defines  $\brcls{\theta}{\kappa}{\crp{F}}$
as an element of $\BrCliff(G,Z)$, 
where $Z=\Z(\crp{F}Ne)$.)

Now suppose that $(\theta,\kappa)$ and 
$(\theta_1,\kappa_1)$ are two semi-invariant 
Clifford pairs over $G$ such that
$\crp{F}(\theta)=\crp{F}(\theta_1)$.
Assume that both Clifford pairs induce
the same action of $G$ on $\crp{F}(\theta)$, that is,
for each $g\in G$ there is $\alpha_g\in \Gal(\crp{F}(\theta)/\crp{F})$
such that $\theta^{g\alpha_g}=\theta$ 
and $\theta_1^{g\alpha_g}=\theta_1$
for all $g\in G$.
If 
\[ \brcls{\theta}{\kappa}{\crp{F}}
     = \brcls{\theta_1}{\kappa_1}{\crp{F}},
\] 
then the character theory of
$\oG$ over the Galois conjugates of $\theta$ 
and the character theory of
$\oG_1$ over the Galois conjugates of $\theta_1$ are essentially 
``the same''~\cite[Theorem~7.12]{turull09},
as we mentioned in the introduction.

\section{Subextensions}
\label{sec:subext}
Let
\[ \begin{tikzcd}
     1 \rar & K \rar & \oG \rar{\kappa} & G \rar & 1 
   \end{tikzcd}
\]
be an exact sequence and $\theta\in \Irr K$.
(Thus $(\theta,\kappa)$ is a Clifford pair.)
Suppose that $\oH$ is a supplement of $K$ in $\oG$,
so that $\oG = \oH K $.
Set $L=\oH\cap K$.
Then $\oH/L \iso \oG/K\iso G$.
Suppose that $\phi\in \Irr L$. 
We want to compare the elements
$\brcls{\theta}{\kappa}{\crp{F}}$
and 
$\brcls{\phi}{\kappa_{|\oH}}{\crp{F}}$.
We do this under additional assumptions,
which we collect here for convenient reference:
\begin{hyp}\label{h:bconf2}
Let
\[
  \begin{tikzcd}
    1 \rar & L \rar \dar[hook] & \oH \rar{\kappa_{|\oH}} \dar[hook]
                               & G \rar \dar[equals] & 1 \\
    1 \rar & K \rar & \oG \rar{\kappa} & G \rar & 1
  \end{tikzcd}
\]
be a commutative diagram of finite groups with exact rows,
let $\theta\in \Irr K $ and $\phi\in \Irr L$ be irreducible
characters and let
$\crp{F}$ be a field of characteristic zero
such that the following conditions hold:
\begin{enums}
\item $n= \ipch{\theta_L}{\phi}>0$.
\item $\crp{F}(\phi)= \crp{F}(\theta) =: \crp{E}$.
\item For every $h\in \oH$ there is
      $\gamma= \gamma_h\in  \Gal( \crp{E} /\crp{F})$ such that
      $\theta^{h\gamma}= \theta$ and
      $\phi^{h\gamma} = \phi$.
\end{enums}
To reduce visual clutter, we write
\[ e = e_{(\theta,\crp{F}) }
     = \sum_{ \gamma \in \Gal( \crp{E} / \crp{F} ) 
            } e_{\theta}^{\gamma}
   \quad \text{and}\quad
   f = e_{(\phi,\crp{F}) }
        = \sum_{ \gamma \in \Gal( \crp{E} / \crp{F} ) 
               } e_{\phi}^{\gamma}      
\]
for the corresponding central primitive idempotents
in $\crp{F}K$ and $\crp{F}L$, respectively.
\end{hyp}
\begin{lemma}[{\cite[Lemma~6.3]{ladisch10pre}}]
\[i :=\sum_{\gamma\in \Gal( \crp{E} / \crp{F})}
e_{\phi}^{\gamma}e_{\theta}^{\gamma}
\]
is a $\oH$\nbd stable nonzero idempotent in 
$\crp{F}Ke$, 
and we have
$ei = i = ie$ and
$fi = i = if$.
\end{lemma}
\begin{lemma}[{\cite[Lemma 6.4]{ladisch10pre}}]
\label{l:centerisos}
  \[ \Z(i\crp{F}Ki)  \iso  \Z(\crp{F}Ke)
                 \iso \crp{F}(\theta) = \crp{F}(\phi)
                 \iso \Z(\crp{F}Lf)
  \]
  as $G$\nbd rings.
\end{lemma}
\begin{lemma}[{\cite[Lemma 6.5]{ladisch10pre}}]
\label{l:StensorFLfiFKi}
Set $Z= \Z(i\crp{F}Ki)$ and
let $S= (i\crp{F}Ki)^L$.
Then $S$ is a central simple $G$\nbd algebra over 
$Z \iso \crp{E} = \crp{F}(\theta)$ with dimension $n^2$ over $Z$,
and
\[\C_{i\crp{F}Ki}(S) = \crp{F}Li \iso \crp{F}Lf.\]
\end{lemma}
Thus the equivalence class of $S$ defines an element
in $\BrCliff(G,\crp{E} )$.
The main result of this section is:
\begin{thm}\label{t:comparesubext}
  Assume Hypothesis~\ref{h:bconf2}.
  Then 
  \[ \brcls{\theta}{\kappa}{\crp{F}} \cdot [S]
      = \brcls{\phi}{\kappa_{|\oH}}{\crp{F}}
      \quad \text{(in $\BrCliff(G,\crp{E})$)}.
  \]
\end{thm}
This result is of course related to the results in~\cite{ladisch10pre},
but we didn't use the language of the Brauer-Clifford 
group in our previous paper.
Thus we give a translation here.
Theorem~\ref{t:comparesubext} is also related to
results in~\cite{turull12,turull13}.

The precise result depends on the following conventions:
All modules over $\crp{F}\oG$ and other group algebras are 
right modules,
and endomorphism rings also operate from the right.
Thus for any $\crp{F}$\nbd algebra $A$ operating
on a module $V$ we get a homomorphism
$A\into \enmo_{\crp{F}}(V)$, both operating from the right.
Then $\enmo_A(V)$ is simply the centralizer of the image of $A$
in $\enmo_{\crp{F}}(V)$.
With other conventions, one may have to 
replace $[S]$ by $[S^{\text{op}}] = [S]^{-1}$ 
in the formula of 
Theorem~\ref{t:comparesubext}.
\begin{proof}[Proof of Theorem~\ref{t:comparesubext}]
  Let $V$ be an $\crp{F}\oG$\nbd module with  
  $Ve=V$. 
  Then the $G$\nbd algebra $\enmo_{\crp{F}K}(V)$
  is in $\brcls{\theta}{\kappa}{\crp{F}}$.
  Moreover, set $W=Vi$.
  Then $W$ is an $\crp{F}\oH$\nbd module with
  $Wf=Vif=Vi=W$, and thus
  $\enmo_{\crp{F}L}(W)$ is a $G$\nbd algebra in 
  $\brcls{\phi}{\kappa_{|\oH}}{\crp{F}}$.
  
  We may also consider $W=Vi$ as a module over
  $i\crp{F}K i$.
  Consider the algebra 
  $A= \enmo_{i\crp{F}Ki}(Vi)$.
  Since $\oH$ acts on $Vi$ and 
  elements of $L$ commute with elements of $A$, 
  we see that $G\iso\oH/L$ acts on $A$.
  We claim that
  \[ A\iso \enmo_{\crp{F}K}(V)
  \]
  as $G$\nbd algebras over $Z$.
  To see this, consider the map $\enmo_{\crp{F}K}(V) \to A$
  sending
  $\phi\in \enmo_{\crp{F}K}(V)$
  to $\phi i\in A$ defined by
  $(vi)(\phi i)=vi\phi i$.
  It is easily verified that this map is an homomorphism
  of algebras.
  To see that it commutes with the action of $G\iso\oH/L$, 
  observe that  for $h\in \oH$, we have
    \[ v \phi^h i  = vh^{-1} \phi h i
       = vh^{-1}\phi i h = v (\phi i)^h.
    \]

  We get the inverse of $\phi\mapsto \phi i$ as follows:
  Since $\crp{F}Ke$ is a simple ring, we have
  $\crp{F}Ke = \crp{F}K i \crp{F}K$.
  Thus there are elements
  $a_{\nu}$, $b_{\nu}\in \crp{F}K$ such that
    $e = \sum_{\nu} a_{\nu}i b_{\nu}$.
  Then the map sending $\psi\in A = \enmo_{i\crp{F}Ki}(Vi)$
  to $\widehat{\psi}\in \enmo_{\crp{F}K}(V)$ defined
       by
       $v\widehat{\psi}= \sum_{\nu} va_{\nu}i \psi b_{\nu}$
  is the inverse of the above map.
  (In fact, it is well known that
  $\enmo_R(V)\iso \enmo_{iRi}(Vi)$
  for idempotents $i$ in a ring $R$ with
  $R=RiR$.)
  The claim is proved.
  
  It follows from the claim that 
  $A = \enmo_{i\crp{F}K i}(Vi) \in \brcls{\theta}{\kappa}{\crp{F}}$.
  To finish the proof it suffices to show that
  $A \tensor_{\crp{E}} S \iso B:= \enmo_{\crp{F}L}(W)$.
  
  Each of $A$, $B$, $i\crp{F}Ki$ and $S$ acts faithfully on
  $W=Vi$ by (right) multiplication, 
  and thus we may identify these algebras
  with subalgebras of
  $T:=\enmo_{\crp{E}}(W)$.
  Clearly, we have
  \[ A, \; S \subseteq B,\]
  and $A$ and $S$ centralize each other.
  Because $A = \C_{T}(i\crp{F}Ki)$ by definition,
  we have $\C_T(A) = i \crp{F}K i$ by the double centralizer
  property~\cite[Theorem~3.15]{farbNA}.
  It follows that
  \[ \C_B(A) = \C_T(A) \cap B
            = i\crp{F}K i \cap \C_T(i\crp{F}Li)
            = S
  \]
  by the definitions of $B$ and $S$.
  Since $A$, $B$ and $S$ are central simple algebras over $\crp{E}$,
  it follows that 
  $B\iso A\tensor_{\crp{E}} S$~\cite[Corollary~3.16]{farbNA}.
\end{proof}
\begin{cor}\label{c:multonesubext}
  Assume Hypothesis~\ref{h:bconf2} with $n=\ipch{\theta_L}{\phi}=1$.
  Then $\brcls{\theta}{\kappa}{\crp{F}}
        = \brcls{\phi}{\kappa_{|\oH}}{\crp{F}}$.  
\end{cor}
\begin{proof}
  Clear since then $S\iso \crp{E}$ yields the trivial element of
  $\BrCliff(G,\crp{E})$.
\end{proof}
The assumptions of the corollary hold in particular when
$\theta_L=\phi\in \Irr L$ and $\crp{F}(\theta)=\crp{F}(\phi)$.
To see another situation  
where Hypothesis~\ref{h:bconf2} holds with $n=1$,
we need a simple lemma, which
is a minor extension of a result of 
Riese and Schmid~\cite[Theorem~1]{rieschm96}.
(The bijection in the following lemma 
also preserves Schur indices over $\crp{F}$,
 but we will not need this fact.)
\begin{lemma}\label{l:cliffcorfields}
  Let $A\nteq \oG$ be a normal subgroup and
  $\tau\in \Irr A$.
  Let $\crp{F}$ be a field and 
  $\Gamma:= \Gal(\crp{F}(\tau)/\crp{F})$.
  Define
  \[ \oH = \{ g\in \oG \mid 
            \exists \alpha\in \Gamma
                          \colon \tau^{g\alpha}=\tau\}.
  \]
  Then induction defines a bijection between
  \[ \bigcup_{\alpha\in \Gamma} 
       \Irr(\oH\mid \tau^{\alpha})
     \quad \text{and} \quad
     \bigcup_{\alpha\in \Gamma}
       \Irr(\oG\mid \tau^{\alpha})
  \]
  commuting with field automorphisms over $\crp{F}$.
  In particular, 
  $\crp{F}(\psi^{\oG})=\crp{F}(\psi)$ for 
  $\psi\in \Irr(\oH\mid \tau)$. 
\end{lemma}
\begin{proof}
  Write $\oG_{\tau}$ to denote the inertia group of $\tau$.
  Clifford correspondence~\cite[Theorem~6.11]{isaCTdov} 
  yields that induction defines bijections
  from 
  $\Irr(\oG_{\tau}\mid \tau)$ 
  onto
  $\Irr(\oH\mid \tau)$, 
  and from
  $\Irr(\oG_{\tau}\mid \tau)$ 
  onto
  $\Irr(\oG\mid \tau)$.
  Thus induction defines a bijection from
  $\Irr(\oH\mid \tau)$ onto $\Irr(\oG\mid\tau)$.
  The same statement holds with $\tau$ replaced by $\tau^{\alpha}$.
  (Of course, we have $\oG_{\tau^{\alpha}}=\oG_{\tau}$.)
  
  It follows that induction yields a surjective map
  \[ 
     \bigcup_{\alpha\in \Gamma} 
            \Irr(\oH\mid \tau^{\alpha})
      \to 
      \bigcup_{\alpha\in \Gamma} 
             \Irr(\oG\mid \tau^{\alpha}) 
  \]
  which is injective when restricted to some
  $\Irr(\oH\mid \tau^{\alpha})$.
  It remains to show that there occurs no collapsing between two
  different sets 
  $\Irr(\oH\mid \tau^{\alpha_1})$ and 
  $\Irr(\oH\mid \tau^{\alpha_2})$.
  So assume that $\psi^{\oG} = \xi^{\oG}=\chi$ for
  $\psi\in \Irr(\oH\mid \tau^{\alpha_1})$
  and  $\xi\in \Irr(\oH\mid \tau^{\alpha_2})$.
  Then $\tau^{\alpha_1}$ and $\tau^{\alpha_2}$ are both 
  irreducible constituents of  $\chi_A$.
  By Clifford's theorem there is $g\in \oG$ with 
  $\tau^{\alpha_1 g}=\tau^{\alpha_2}$
  or $\tau^{g \alpha_1 \alpha_2^{-1}}=\tau$.
  By the very definition of $\oH$ we see that $g \in \oH$.
  But then $\psi= \psi^{g}\in \Irr(\oH\mid \tau^{\alpha_2})$.
  Since induction is bijective when restricted to
  $\Irr(\oH\mid \tau^{\alpha_2})$, it follows that
  $\psi=\xi$ as wanted.
  
  Obviously, induction commutes with field automorphisms:
  We have $(\psi^{\alpha})^{\oG} = (\psi^{\oG})^{\alpha}$.
  We always have $\crp{F}(\psi^{\oG})\subseteq \crp{F}(\psi)$.
  If $\alpha\in \Gal(\crp{F}(\psi)/\crp{F}(\psi^{\oG}))$,
  then $\psi^{\oG} = (\psi^{\oG})^{\alpha}= (\psi^{\alpha})^{\oG}$
  and thus $\psi=\psi^{\alpha}$ by bijectivity.
  It follows $\alpha=1$ and thus
  $\crp{F}(\psi)=\crp{F}(\psi^{\oG})$ as claimed.  
  The proof is finished.
\end{proof}
\begin{cor}\label{c:subextclifford}
  Let 
  \[ \begin{tikzcd}
       1 \rar & K \rar & \oG \rar{\kappa} & G \rar & 1 
     \end{tikzcd}
  \]
  be an exact sequence of finite groups and
  $\theta\in \Irr K$ be semi-invariant
  over $\crp{F}$ in $\oG$.
  Assume there is $A\nteq \oG$ with $A\subseteq K$
  and let $\tau $ be an irreducible constituent of
  $\theta_A$.
  Set 
  \[ \oH = \{ g\in \oG 
              \mid
              \exists \alpha\in \Gal(\crp{F}(\tau)/\crp{F})
              \colon \tau^{g\alpha}=\tau\}.
  \]
  Then $\oG = \oH K$, and for $L=\oH\cap K$ there is a unique
  $\phi\in \Irr(L\mid \tau)$
  with $\theta=\phi^K$,
  and we have
  $\crp{F}(\theta)=\crp{F}(\phi)$ and
  $\brcls{\theta}{\kappa}{\crp{F}}
          = \brcls{\phi}{\kappa_{|\oH}}{\crp{F}}$
  for this $\phi$.
\end{cor}
\begin{proof}
  We first show that $\oG=\oH K$.
  Let $g\in \oG$.
  Since $\theta$ is semi-invariant in $\oG$,
  there is an $\alpha \in \Gal( \crp{F}(\theta)/\crp{F} ) $
  such that $\theta^{g\alpha}=\theta$.
  We may extend $\alpha$ to an automorphism~$\beta$ of
  $\crp{F}(\theta,\tau)$.
  Then $\tau^{g\beta}$ is a constituent of
  $\theta_A$. 
  By Clifford's theorem, we have 
  $\tau^{g\beta k}=\tau$ for some $k\in K$.
  Thus $gk\in \oH$, which shows $\oG = \oH K$.
  
  Lemma~\ref{l:cliffcorfields} yields the existence of
  a $\phi$ with $\theta=\phi^K$.
  
  We verify that Hypothesis~\ref{h:bconf2} holds.
  We have $n= \ipch{\theta_L}{\phi}=\ipch{\theta}{\phi^K}=1$.
  Lemma~\ref{l:cliffcorfields} yields that
  $\crp{F}(\phi)=\crp{F}(\theta)$.
  Finally, let $h\in \oH$.
  Then there is $\alpha\in \Gal(\crp{F}(\theta)/\crp{F})$
  with $\theta^{h\alpha}=\theta$.
  Then $(\phi^{h\alpha})^K= \theta=\phi^K$.
  Since 
  $\phi^{h\alpha}\in \Irr(L\mid \tau^{h\alpha})$
  and $\tau^{h\alpha}= \tau^{\beta}$ for some
  field automorphism $\beta$,
  it follows from Lemma~\ref{l:cliffcorfields} that
  $\phi^{h\alpha}=\phi$.
  
  We have now shown that Hypothesis~\ref{h:bconf2}
  holds with $n=1$. 
  Corollary~\ref{c:multonesubext} yields the result.
\end{proof}
\begin{remark}
  It follows directly from Lemma~\ref{l:cliffcorfields}
  that for each subgroup 
  $U$ with $K\leq U\leq \oG$ there is a correspondence between
  $\Irr(U\mid \tau)$ and $\Irr(U\cap \oH\mid \tau)$.
  It is also elementary to prove that these correspondences commute
  with restriction and induction of characters.
  In fact, every property of the correspondence that follows
  from the equality 
  $\brcls{\theta}{\kappa}{\crp{F}}
            = \brcls{\phi}{\kappa_{|\oH}}{\crp{F}}$
  and the results of Turull~\cite{turull09}
  can be proved elementarily, 
  without using the Brauer-Clifford group.
  But we will need Corollary~\ref{c:subextclifford}
  in inductive arguments to come later,
  and there we can not do without the language 
  of the Brauer-Clifford group.
\end{remark}
Proposition~\ref{p:subgroupsmain} is a special case of
Corollary~\ref{c:subextclifford}, applied 
over a bigger field:
\begin{proof}[Proof of Proposition~\ref{p:subgroupsmain}]
Assume the situation of Theorem~\ref{main}.
Recall that $C$, $N_1 \nteq \oG_1$ with $C\subseteq N_1$,
and that $\theta_1 = \lambda^{N_1}$ with
$\lambda\in \Irr C$.
By definition in Proposition~\ref{p:subgroupsmain},
$\widehat{U}_1 =(\oG_1)_{\theta_1}$ is the inertia group of 
$\theta_1$ in $\oG_1$,
and $V=(\oG_1)_{\lambda}$ is the inertia group of $\lambda$ in $\oG_1$.
Clearly, $\lambda^{N_1}=\theta_1\in \Irr N_1$ 
implies that $V\cap N_1 = C$.

We apply Corollary~\ref{c:subextclifford} to the exact sequence
\[ \begin{tikzcd}
     1 \rar & N_1 \rar & \widehat{U}_1 \rar & U \rar & 1
   \end{tikzcd}
\]
with $C$ instead of $A$, $\lambda$ instead of $\tau$ 
and the field $\crp{F}(\lambda)$ instead of $\crp{F}$.
Note that $\oH = V$.
We get that $\widehat{U}_1 = VN_1$
and $\brcls{\theta_1}{(\kappa_1)_{|\widehat{U}_1}}{\crp{F}(\lambda)}
    = \brcls{\lambda}{(\kappa_1)_{|V}}{\crp{F}(\lambda)}$.
By Lemma~\ref{l:cliffcorfields},
induction induces a bijection
$\Irr(X\cap V \mid \lambda) \to \Irr( X \mid \theta_1)$.
\end{proof}

\section{A subgroup of the Schur-Clifford group}
\label{sec:cycltoabel}
Let $\crp{E}$ be a field and $G$ a group which acts on $\crp{E}$
and fixes the subfield $\crp{F}$.
Recall that in~\cite{ladisch15b}            
we defined the \defemph{Schur-Clifford group}
$\SC^{(\crp{F})}(G,\crp{E})$ to be the subset
of $\BrCliff(G,\crp{E})$
of all $ \brcls{\theta}{\kappa}{\crp{F}}$
such that $(\theta,\kappa)$ is a Clifford pair over $G$
that induces the given action of $G$ on $\crp{F}(\theta)=\crp{E}$.
We also showed that $\SC^{(\crp{F})}(G,\crp{E})$ is a subgroup
of the Brauer-Clifford group, if $\crp{E}$ is contained in a
cyclotomic extension of $\crp{F}$.
Further, in that case we have 
\[ \SC^{(\crp{F})}(G,\crp{E})
    = \SC^{(\crp{E}^G)}(G,\crp{E})
    =: \SC(G,\crp{E}).
\]

Fix a finite group $G$ and a field $\crp{F}$.
We consider two subclasses
of Clifford pairs
$(\theta,\kappa)$
where, as usual, the homomorphism
$\kappa$ is part of an exact sequence
\[\begin{tikzcd}[ampersand replacement =\&]
                1 \rar \& N \rar \& \oG \rar{\kappa} \& G \rar \& 1
   \end{tikzcd}  
\]
and $\theta\in \Irr N$.
The main result of this section is that both subclasses yield the same 
subgroup of the Schur-Clifford group.
\begin{defi}\label{defi:subclasses}
Let $\macl$ be the class of all Clifford pairs 
$(\theta,\kappa)$ such that
\begin{enums}
\item \label{i:cond_abel}
      there is $A\nteq \oG$ with $A\subseteq N$, 
\item \label{i:cond_induc}
      there is a linear character $\lambda\in \Lin A$ such that 
      $\theta=\lambda^N$.
\end{enums}
Let $\cyccl$ be the subclass of $\macl$ 
containing the pairs in $\macl$ such that 
\begin{enums}[resume]
\item the $A$ above is cyclic and $\lambda$ is faithful,
      \label{i:cond_cyclic}
\item $\lambda$ is semi-invariant in $\oG$ over $\crp{F}$, and
      \label{i:cond_semiinv}      
\item  $N/A\iso \Gal(\crp{F}(\lambda)/\crp{F}(\theta))$.
       \label{i:cond_galiso}
\end{enums}
\end{defi}
Recall that for a class $\mathcal{X}$ of Clifford pairs, we defined
\[ \SC_{\mathcal{X}}(G,\crp{E})
    = \{ \brcls{\theta}{\kappa}{\crp{F}}\in \BrCliff(G,\crp{E})
         \mid 
         (\theta,\kappa) \in \mathcal{X}
      \}.
\]
In this terminology,  Theorem~\ref{main} says
that $\SC(G,\crp{E})=\SC_{\cyccl}(G,\crp{E})$.

There is some redundancy in these conditions:
\begin{lemma}\label{l:galisoredundant}
  If~\ref{i:cond_abel}, \ref{i:cond_induc}
  and~\ref{i:cond_semiinv} of Definition~\ref{defi:subclasses} 
  hold for some Clifford pair
  $(\theta,\kappa)$,
  so does~\ref{i:cond_galiso}.
  (Thus Condition~\ref{i:maingaliso} in Theorem~\ref{main} 
  follows from the other conditions
  in Theorem~\ref{main}.)
\end{lemma}
\begin{proof}
   Since $\lambda$ is semi-invariant in $\oG$ over
   $\crp{F}$, 
   there is for every $n\in N$
   an $\alpha_n\in \Gal(\crp{F}(\lambda)/\crp{F})$
   such that $\lambda^{n\alpha_n}=\lambda$.
   Since $(\lambda^n)^N= \theta^n=\theta$, we must have
   $\alpha_n \in \Gal(\crp{F}(\lambda)/\crp{F}(\theta))$.
    This defines an homomorphism
    from $N/A$ into 
    $\Gal(\crp{F}(\lambda)/\crp{F}(\theta))$ with kernel
    $N_{\lambda}/A$~\cite[Lemma~2.1]{i81b}.
    Since $\lambda^{N}\in \Irr N$, 
    we must have $N_{\lambda} = A$.
    
    Now let $\alpha\in \Gal(\crp{F}(\lambda)/\crp{F}(\theta))$.
    Then $\lambda^{\alpha}$ and $\lambda$ are constituents
    of $\theta=\theta^{\alpha}$ and thus conjugate
    in $N$, so that
    $\lambda^{\alpha n}= \lambda$ for some $n\in N$.
    It follows that 
    $\alpha=\alpha_n\in \Gal(\crp{F}(\lambda)/\crp{F}(\theta))$.
    Thus $n\mapsto \alpha_n\in \Gal(\crp{F}(\lambda)/\crp{F}(\theta))$
    is surjective.
    Thus  Condition~\ref{i:cond_galiso} holds.
\end{proof}
The following result follows directly from the definitions.
\begin{lemma}\label{l:cliffpairkernel}
  Let $(\theta, \kappa\colon \oG\to G)$ be a Clifford pair and 
  $K\subseteq \Ker\theta$ a normal subgroup of $\oG$.
  Then 
  $\brcls{\overline{\theta}}{\overline{\kappa}}{\crp{F}}
     = \brcls{\theta}{\kappa}{\crp{F}}$,
  where $\overline{\kappa}\colon \oG/K\to G$
  is the map induced by $\kappa$ and
  $\overline{\theta}\in \Irr(N/K)$ is the character
  $\theta$ viewed as character of $N/K$
  (where $N=\Ker\kappa$).
\end{lemma}
Note that when Conditions~\ref{i:cond_abel}, \ref{i:cond_induc}
  and~\ref{i:cond_semiinv} above hold,
then $K=\Ker\lambda$ is normal in $\oG$, and 
also $K=\Ker \theta$.
Thus we can factor out $K$ and get a Clifford pair 
such that \ref{i:cond_cyclic} is true, too.
\begin{prop}\label{p:cycltoabel}
  Let\/ $\crp{E}$ be a field extension of\/ $\crp{F}$
  on which $G$ acts as $\crp{F}$\nbd algebra.
  Then
  \[ \SC_{\macl}(G,\crp{E})= \SC_{\cyccl}(G,\crp{E}). 
  \]
\end{prop}
\begin{proof}
  By definition, we have
  $\SC_{\cyccl}(G,\crp{E})\subseteq \SC_{\macl}(G,\crp{E})$.
  To show the converse inclusion,
  we begin with a Clifford pair
  $(\theta,\kappa)\in \macl$
  that induces the given action on $\crp{E}=\crp{F}(\theta)$.
  Here $\kappa\colon \oG\to G$ and $\theta\in \Irr N$
  with $N=\ker\kappa$ as usual.
  Let $A\subseteq N$ be a normal subgroup of $\oG$
  and $\lambda\in \Lin A$ a linear character with
  $\theta=\lambda^N$.
  
  By assumption,  
  the character $\theta$ is semi-invariant in $\oG$.
  Set
  \begin{align*}
    \oG_0 &= \{ g\in \oG \mid
              \quad\text{there is
              $\alpha\in \Gal(\crp{F}(\lambda)/\crp{F})$
              such that $\lambda^{g\alpha}= \lambda$}\},
              \\
    N_0 &= N\cap \oG_0,\quad \kappa_0=\kappa_{|\oG_0},
           \quad\text{and}\quad 
           \theta_0= \lambda^{N_0}.
  \end{align*}
  Note that $\theta_0^N = \lambda^N = \theta$.
  By Corollary~\ref{c:subextclifford},
  we have that $\oG=\oG_0 N$, that
  $\crp{F}(\theta)=\crp{F}(\theta_0)$,
  and that 
  $\brcls{\theta}{\kappa}{\crp{F}} 
    = \brcls{\theta_0}{\kappa_0}{\crp{F}}$.
  Condition~\ref{i:cond_semiinv} (and thus~\ref{i:cond_galiso})
  holds for the Clifford pair $(\theta_0, \kappa_0)$.
  By Lemma~\ref{l:cliffpairkernel} and the remark following it, 
  we can factor out
  the kernel of $\lambda$ and we get a
  Clifford pair in $\cyccl$ that yields the same
  element of the Brauer-Clifford group as $(\theta,\kappa)$.  
\end{proof}
\begin{cor}\label{c:sccgroup}
  $\SC_{\cyccl}(G,\crp{E})$ is a subgroup of $\SC(G,\crp{E})$
   and $\BrCliff(G,\crp{E})$.
\end{cor}
\begin{proof}
  It suffices to show that
  $\SC_{\macl}(G,\crp{E})$ is a subgroup.
  By Theorem~5.7 from~\cite{ladisch15b}, 
  we have to show the following:
  \begin{enums}
    \item $\SC_{\macl}(G,\crp{E})\neq \emptyset$.
    \item \label{i:invclosed} When $(\theta,\kappa)\in \macl$,
          then $(\cconj{\theta},\kappa)\in \macl$
          (here, $\cconj{\theta}$ denotes the complex conjugate
          of $\theta$).
    \item \label{i:prodclosed} When $(\theta_1,\kappa_1)$ and 
          $(\theta_2,\kappa_2)\in \macl$,
          then $(\theta_1\times \theta_2,\kappa_1\times_G\kappa_2)
          \in \macl$.
  \end{enums}   
  That $\SC_{\macl}(G,\crp{E})\neq \emptyset$ 
  follows from Corollary~5.2 in \cite{ladisch15b}, 
  and \ref{i:invclosed} is clear.
  For \ref{i:prodclosed}, 
  assume $A_i\subseteq N_i$ is the abelian normal subgroup
  of $\oG_i$, 
  where 
  \[\begin{tikzcd}
         1 \rar & N_i \rar & \oG_i \rar{\kappa_i} & G \rar & 1
    \end{tikzcd}
  \]
  is the exact sequence belonging to $\kappa_i$,
  and let $\lambda_i\in \Lin A_i$ 
  with $\theta_i= \lambda_i^{N_i}$ for 
  $i=1$,~$2$.
  Recall that the Clifford pair 
  $(\theta_1\times \theta_2,\kappa_1\times_G\kappa_2)$
  is defined as follows:  
    Let 
    \[ \oG_1 \times_G \oG_2 = \{(g_1,g_2) \in \oG_1\times\oG_2
                                 \mid 
                                 g_1\kappa_1 = g_2\kappa_2
                              \}
    \]
    be the pullback and
    $\kappa_1\times_G \kappa_2\colon \oG_1 \times_G \oG_2 \to G$ 
    the canonical homomorphism mapping
    $(g_1,g_2)$ to $g_1\kappa_1 = g_2\kappa_2$.
    Then $\Ker(\kappa_1\times_G\kappa_2) = N_1\times N_2$,
    and $\theta_1\times\theta_2\in \Irr(N_1 \times N_2)$.
  It follows that
  $A_1\times A_2\nteq \oG_1 \times_G \oG_2$ 
  is an abelian normal subgroup and
  $\theta_1 \times \theta_2 =
   \lambda_1^{N_1}\times\lambda_2^{N_2}
   = (\lambda_1\times\lambda_2)^{N_1\times N_2}$.   
  This shows~\ref{i:prodclosed}.
\end{proof}

\section{Reduction to prime power groups}
\label{sec:red_pgrp}
\begin{lemma}\label{l:reducppart}
  Let $a\in \SC(G,\crp{E})$.
  Then $a\in \SC_{\cyccl}(G,\crp{E})$
  if and only if the $p$-parts $a_p$
  of $a$ 
  are in $\SC_{\cyccl}(G,\crp{E})$ for all primes $p$.
\end{lemma}
\begin{proof}
  Recall that $\BrCliff(G,\crp{E})$ is torsion.
  Thus $a= \prod_{p} a_p$
  is the product of its $p$-parts $a_p$, 
  and  $a_p\in \erz{a}$.
  The result follows since 
  $\SC(G,\crp{E})$ and $\SC_{\cyccl}(G,\crp{E})$
  are subgroups (Corollary~\ref{c:sccgroup}).
\end{proof}
For any subgroup $H\leq G$, there is a group homomorphism
\[ \Res_H^G\colon \BrCliff(G,\crp{E})\to \BrCliff(H,\crp{E})
\]
which is induced by viewing a $G$\nbd algebra as an 
$H$\nbd algebra.
This restriction homomorphism sends 
$\brcls{\theta}{\kappa}{\crp{F}}$ to
$\brcls{\theta}{\pi}{\crp{F}}$, 
where $\pi$ is the restriction of $\kappa$ to the preimage
$\kappa^{-1}(H)$ of $H$~\cite[Proposition~7.1]{ladisch15b}.

In the proof of the next result, we also need the corestriction
map
\[ \Cores_H^G \colon \BrCliff(H,\crp{E})\to \BrCliff(G,\crp{E})
\]
defined in~\cite{ladisch15a}. 
\begin{lemma}\label{l:reducpgroup}
  Let $a\in \SC(G,\crp{E})$ have
  $p$-power order for the prime $p$ and let $P\leq G$ be a
  Sylow $p$\nbd subgroup of $G$.
  Then $a\in \SC_{\cyccl}(G,\crp{E})$
  if and only if
  $\Res_P^G(a)\in \SC_{\cyccl}(P,\crp{E})$.
\end{lemma}
\begin{proof}
  The ``only'' if part is clear and does not depend on
  $a\in \SC(G,\crp{E})$ having
  $p$\nbd power order.
  
  Now assume 
  $\Res_P^G(a)\in \SC_{\cyccl}(P,\crp{E})$.
  Apply the corestriction map
  \[ \Cores_P^G \colon \BrCliff(P,\crp{E}) \to \BrCliff(G,\crp{E}).
  \]
  It follows from \cite[Corollary~9.2]{ladisch15b} 
  that $\Cores_P^G$ maps 
  $\SC_{\macl}(P,\crp{E})$ into $\SC_{\macl}(G,\crp{E})$.
  Since $\SC_{\macl}(G,\crp{E})= \SC_{\cyccl}(G,\crp{E})$
  by Proposition~\ref{p:cycltoabel},
  we have
  $\Cores_P^G(\Res_P^G(a)) \in \SC_{\cyccl}(G,\crp{E})$.  
  By \cite[Theorem~6.3]{ladisch15a}, 
  we get
  $\Cores_P^G(\Res_P^G(a))= a^{\abs{G:P}}$.
  Since $a$ has $p$\nbd power order, it follows that
  $a\in \erz{a^{\abs{G:P}}}$.
  Thus $a\in \SC_{\cyccl}(G,\crp{E})$ as claimed.
\end{proof}
By the last two results, to show that an arbitrary element
$a\in \SC(G,\crp{E})$ 
is in fact contained in $\SC_{\cyccl}(G,\crp{E})$,
we can assume that $a$ has $p$\nbd power order and 
that $G$ is a $p$-group.
Note that it was essential in the above proofs that
$\SC(G,\crp{E})$ and $\SC_{\cyccl}(G,\crp{E})$ are groups.
\begin{remark}
  By similar arguments, one can show:
  An element $a\in \BrCliff(G,\crp{E})$ is in
  $\SC_{\cyccl}(G,\crp{E})$ if and only if
  $\Res_P^G(a)\in \SC_{\cyccl}(P,\crp{E})$
  for all Sylow subgroups $P$ of $G$.
  Note, however, that even when $P$ is a $p$-group,
  the exponent of
  $\BrCliff(P,\crp{E})$ or $\SC(P,\crp{E})$
  is in general not a $p$\nbd power.
  For example, $\BrCliff(1,\crp{E})=\Br(\crp{E})$,
  and $\SC(1,\crp{E})$ is in general not the trivial group. 
\end{remark}

\section{Reduction to a larger field}
In this section, we show that when $G$ is a $p$\nbd group
and $\brcls{\theta}{\kappa}{\crp{F}}$ has 
$p$\nbd power-order, then we can 
replace the fields $\crp{E}$ and $\crp{F}$ by certain larger fields
to prove the main theorem (Theorem~\ref{main}).

First, we assume the following situation:
Let $\crp{E}$ be a field on which a group $G$ acts, and let
$\crp{F}$ be a field contained in the fixed field
$\crp{E}^G$.
Suppose that there are two other
fields $\crp{K}\geq \crp{L}\geq \crp{F}$
such that $\crp{K}= \crp{E}\crp{L}$
and $\crp{E}\cap \crp{L}= \crp{F}$.
Assume that 
$\crp{L}/\crp{F}$ is Galois and let
$\Delta= \Gal(\crp{L}/\crp{F})\iso 
      \Gal(\crp{K}/\crp{E})$.
The situation is summarized in the following picture:
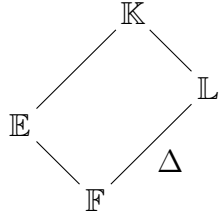
\begin{center}
  \begin{tikzpicture}[on grid=true, inner sep=2pt]
             \node (F) {$\crp{F}$};
             \node (L) [above right=2.1 of F] {$\crp{L}$};
             \node (E) [above left=1.4 of F] {$\crp{E}$};
             \node (K) [above left=1.4 of L] {$\crp{K}$};
             \draw (F) -- (L) node[auto=right, midway] {$\Delta$};
             \draw (L) -- (K) -- (E) -- (F);
  \end{tikzpicture}
\end{center}
In this situation, 
$G\times \Delta$ acts on 
$\crp{K}\iso \crp{E}\tensor_{\crp{F}} \crp{L}$.
If $S$ is a central simple $G$\nbd algebra over $\crp{E}$, 
then $S\tensor_{\crp{F}}\crp{L} \iso S\tensor_{\crp{E}}\crp{K}$
is a central simple 
$(G\times\nobreak \Delta)$\nbd algebra over $\crp{K}$.
This defines a group homomorphism
$\BrCliff(G,\crp{E})\to \BrCliff(G\times\nobreak\Delta,\crp{K})$.
In fact, this group homomorphism is the composition of
\[ 
  \begin{tikzcd}
    \BrCliff(G,\crp{E})\rar{\Inf}
      & \BrCliff(G\times \Delta,\crp{E}) \rar
      & \BrCliff(G\times \Delta, \crp{K}),
  \end{tikzcd}
\]
where the first map is the inflation map induced by the 
epimorphism $G\times\Delta\to G$, and the second
map is induced by scalar extension from
$\crp{E}$ to $\crp{K}$.
Our first goal is to show that the above group homomorphism
is actually an isomorphism:
\begin{prop}\label{p:fieldred}
  The maps 
  \[ S\mapsto S\tensor_{\crp{E}}\crp{K}
     \quad 
     \text{and}
     \quad
     T \mapsto T^{\Delta} := \C_T(\Delta)
  \]
  define inverse maps between
  the isomorphism classes of central simple
  $G$\nbd algebras over $\crp{E}$ and central simple
  $(G\times\Delta)$\nbd algebras over $\crp{K}$.
  Moreover, they induce mutually inverse
  isomorphisms
  \[ \BrCliff(G,\crp{E})
      \iso \BrCliff(G\times \Delta, \crp{K}).
  \]
\end{prop}
It is clear that $S\iso (S \tensor_{\crp{E}} \crp{K})^{\Delta}$,
and it is a result of Hochschild~\cite[Lemma~1.2]{hochs50} that
$T\iso T^{\Delta}\tensor_{\crp{E}}\crp{K}$. 
But to see that the inverse map sending $T$ to $T^{\Delta}$
respects equivalence classes, we need some more general arguments,
and it will be more convenient for us to reprove
Hochschild's result.

Let $\crp{K}\Delta$ denote the skew group ring with respect
to the action of $\Delta$ on $\crp{K}$
(see Section~\ref{sec:review}).
The ring $\crp{K}\Delta$ acts on $\crp{K}$ from the right by
\[x\circ \sum_{\sigma\in \Delta} \sigma c_{\sigma}
   = \sum_{\sigma} x^{\sigma}c_{\sigma}.
\]
This makes $\crp{K}$ into a right $\crp{K}\Delta$\nbd module.
So if $V$ is a vector space over $\crp{E}$, 
then $V\tensor_{\crp{E}}\crp{K}$ is a right $\crp{K}\Delta$\nbd module.
For a right $\crp{K}\Delta$\nbd module $W$,
we still write 
$ W^{\Delta} 
  = \{w\in W 
      \mid 
      w\sigma = w 
      \text{ for all }
      \sigma\in \Delta
    \}$.

The next lemma basically follows from the 
fact that 
$\crp{K}\Delta\iso \mat_{\abs{\Delta}}(\crp{E})$
which is well known from Galois cohomology.
\begin{lemma}\label{l:fieldcatequiv}
  For every $\crp{E}$\nbd vector space $V$ and every
  right $\crp{K}\Delta$\nbd module $W$, we have  
  \[ V \iso (V\tensor_{\crp{E}}\crp{K})^{\Delta}
       \quad 
       \text{and}
       \quad
       W \iso  W^{\Delta} \tensor_{\crp{E}}\crp{K}
  \]
  naturally, and 
  $\enmo_{\crp{K}\Delta}(W) \iso \enmo_{\crp{E}}(W^{\Delta})$. 
  (In fact, we have a category equivalence between
  the category of vector spaces over $\crp{E}$
  and the category of modules over $\crp{K}\Delta$.)
\end{lemma}
\begin{proof}
  It is clear that $(V\tensor_{\crp{E}}\crp{K})^{\Delta}\iso V$
  naturally for any $\crp{E}$\nbd vector space $V$.
  
  Conversely, let $W$ be a $\crp{K}\Delta$\nbd module.
  We have to show that $w\tensor k\mapsto wk$
  is an isomorphism $W^{\Delta}\tensor_{\crp{E}}\crp{K}\iso W$.
  
  First, we observe the following identity in $\crp{K}\Delta$:
  For any $a\in \crp{K}$, we have
  \[ \big(\sum_{\sigma\in \Delta}\sigma \big)
     \,  a  \,
     \big(\sum_{\tau\in \Delta}\tau \big)
     = \sum_{\sigma, \tau\in \Delta}
       \sigma \tau a^{\tau}
     = \big( \sum_{\sigma\in \Delta} \sigma \big)
       \Tr_{\crp{E}}^{\crp{K}}(a).
  \]
  
  Let $b_1$, $\dotsc$, $b_n$ be a basis of $\crp{K}$
  over $\crp{E}$ 
  and let $a_1$, $\dotsc$, $a_n$ be the dual basis with respect to 
  the form $(x,y)\mapsto \Tr_{\crp{E}}^{\crp{K}}(xy)$,
  so that 
  $\Tr_{\crp{E}}^{\crp{K}}(b_ia_j) = \delta_{ij}$.
  Then for all $k\in \crp{K}$ we have
  \[ k = \sum_{i=1}^n \Tr_{\crp{E}}^{\crp{K}}(ka_i)b_i
       = \sum_{i=1}^n \Tr_{\crp{E}}^{\crp{K}}(kb_i)a_i.
  \]
  Set 
  \[ E_{ij} := a_i \big(\sum_{\sigma\in \Delta} \sigma \big) b_j 
     \in \crp{K}\Delta.
  \]
  It follows that 
  \[ E_{ij}E_{rs} = a_i \big(\sum_{\sigma\in \Delta} \sigma \big)
                    \Tr_{\crp{E}}^{\crp{K}}(b_ja_r) \, b_s
                  = \delta_{jr} E_{is},                  
  \]
  where the first equality follows from the identity in the last 
  paragraph.
  In particular, the $E_{ij}$'s are linearly independent 
  over $\crp{E}$
  and so, by counting dimensions, form a basis of $\crp{K}\Delta$.
  Since $(\sum_{i}E_{ii} ) E_{rs} = E_{rs}$,
  it follows that $\sum_i E_{ii} = 1$.
  We have now proved that the $E_{ij}$'s form a 
   \emph{complete set of matrix units}.
  Using this,
  it is routine to verify that
  \[ W \ni w \mapsto 
     \sum_{i=1}^n w a_i \big(\sum_{\sigma\in \Delta} \sigma \big)
     \tensor b_i \in W^{\Delta} \tensor_{\crp{E}} \crp{K}
  \]
  is the inverse of the natural map
  $W^{\Delta}\tensor_{\crp{E}}\crp{K} \to W$
  sending $w_0\tensor k$ to $wk$.
  
  The isomorphism 
  $\enmo_{\crp{K}\Delta}(W) \iso \enmo_{\crp{E}}(W^{\Delta})$
  follows from
  $W\iso W^{\Delta}\tensor_{\crp{E}} \crp{K}$.
\end{proof}

\begin{lemma}\label{l:fieldcentprod}
  Let $W_1$ and $W_2$ be two $\crp{K}\Delta$\nbd modules.
  Then
  $(W_1\tensor_{\crp{K}}W_2)^{\Delta}
    \iso 
    W_1^{\Delta}\tensor_{\crp{E}} W_2^{\Delta}$.
\end{lemma}
\begin{proof}
  We have a natural injection of
  $W_1^{\Delta}\tensor_{\crp{E}} W_2^{\Delta}$
  into $(W_1\tensor_{\crp{K}}W_2)^{\Delta}$.
  It follows from Lemma~\ref{l:fieldcatequiv}
  that both spaces have dimension
  $(\dim_{\crp{K}}W_1)(\dim_{\crp{K}}W_2)$ over $\crp{E}$.
  Thus the injection is an isomorphism.
\end{proof}
\begin{proof}[Proof of Proposition~\ref{p:fieldred}]
  A $(G\times \Delta)$\nbd algebra over $\crp{K}$ can be viewed as a
  $\crp{K}\Delta$\nbd module.
  It is easy to check that the isomorphisms of 
  Lemma~\ref{l:fieldcatequiv}
  are isomorphisms of 
  $(G\times\Delta)$\nbd algebras.
  
  If $S_1$ and $S_2$ are equivalent, then 
  $S_1\tensor_{\crp{E}}\crp{K}$ and
  $S_2\tensor_{\crp{E}}\crp{K}$ are equivalent, 
  as explained at the beginning of this section.
  
  Now assume that $[T_1]=[T_2]$ in $\BrCliff(G\times\Delta,\crp{K})$.
  To finish the proof, we need to show that
  $T_1^{\Delta}$ and $T_2^{\Delta}$ are equivalent.
  By assumption, there are $\crp{K}[G\times\Delta]$\nbd modules 
  $P_1$ and $P_2$ such that
  \[ T_1\tensor_{\crp{K}}\enmo_{\crp{K}}(P_1)
     \iso
     T_2 \tensor_{\crp{K}}\enmo_{\crp{K}}(P_2).
  \]
  Taking centralizers of $\Delta$ and using Lemma~\ref{l:fieldcentprod}, 
  it follows that
  \[ T_1^{\Delta} \tensor_{\crp{E}}\enmo_{\crp{K}\Delta}P_1
     \iso 
     T_2^{\Delta} \tensor_{\crp{E}}\enmo_{\crp{K}\Delta}P_2.
  \]
  (Note that $(\enmo_{\crp{K}}P_i)^{\Delta}=\enmo_{\crp{K}\Delta}P_i$.)
  From Lemma~\ref{l:fieldcatequiv}, it follows that
  $\enmo_{\crp{K}\Delta}P_i
   \iso \enmo_{\crp{E}}(P_i^{\Delta})$.
  The centralizer $P_i^{\Delta}$ is an $\crp{E}G$\nbd submodule
  of $P_i$, 
  and the isomorphisms above are isomorphisms of $G$\nbd algebras.
  Thus $\enmo_{\crp{E}}(P_i^{\Delta})$ is a trivial $G$\nbd algebra
  over~$\crp{E}$ (for $i=1$,~$2$),
  and $T_1^{\Delta}$ and $T_2^{\Delta}$ are equivalent.
\end{proof}
We write 
\[ C(\Delta)\colon \BrCliff(G\times\Delta,\crp{K})
             \to \BrCliff(G, \crp{E})
    ,\quad
   C(\Delta)[T]=[T^{\Delta}].
\]
\begin{lemma}\label{l:galoiscentrsc}
  We have 
  \begin{align*}
   C(\Delta)(\SC(G\times \Delta,\crp{K}))
   &\subseteq \SC(G, \crp{E})
   \quad\text{and}\\
   C(\Delta)(\SC_{\macl}(G\times \Delta,\crp{K}))
     &\subseteq \SC_{\macl}(G, \crp{E}).
  \end{align*}
\end{lemma}
\begin{proof}
  Let 
  \[ \begin{tikzcd}
        1 \rar & N \rar & \oG \rar{\kappa}
               & G \times \Delta \rar
               &1
     \end{tikzcd}
  \]
  be an exact sequence and $V$ an $\crp{F}\oG$\nbd module
  with $S=\enmo_{\crp{F}N}V$, and such that the class of $S$ is in
  $\SC(G\times\Delta,\crp{K})$.
  Then for $M= \kappa^{-1}(\Delta)$
  we have an exact sequence
  \[ \begin{tikzcd}
       1\rar & M \rar & \oG \rar{\pi}
             & G \rar & 1,
    \end{tikzcd}
  \]
  where $\pi $ is
  $\begin{tikzcd}
     \oG \rar{\kappa} & G\times\Delta \rar & G.
  \end{tikzcd}$
  The equality
  $\enmo_{\crp{F}M}(V)= S^{\Delta}$
  proves the first assertion.
  
  Let $\theta\in \Irr N$ be an irreducible constituent of the character
  of $V_{N}$.
  Then $\crp{F}(\theta)= \crp{K}\iso\Z(S)$.
  Since $\Delta$ acts faithfully on $\crp{K}$,
  it follows that $\theta^m \neq \theta$ for all
  $m\in M\setminus N$.
  Thus $\theta^M$ is irreducible, and it follows
  that
  $[S^{\Delta}]= \brcls{\theta^M}{\pi}{\crp{F}}$.
  This shows, in particular, the second assertion.
\end{proof}
\begin{lemma}\label{l:fieldpowerred}
  Let $[S]\in \BrCliff(G, \crp{E})$.
  If $[S\tensor_{\crp{F}}\crp{L}]\in \SC_{\macl}(G,\crp{K})$, 
  then $[S]^{\abs{\Delta}}\in \SC_{\macl}(G, \crp{E})$.
\end{lemma}
\begin{proof}
  Note that we can view 
  $S\tensor_{\crp{F}}\crp{L}$ as a
  $(G\times\Delta)$\nbd algebra.
  Applying the corestriction map
  with $G\leq G\times \Delta$, we get~\cite[Theorem~6.3]{ladisch15a}
  \[ \Cores_G^{G\times \Delta}\Res_G^{G\times \Delta}
       [S\tensor_{\crp{F}}\crp{L}]
     = [S\tensor_{\crp{F}}\crp{L}]^{\abs{\Delta}}.
  \]
  Since we assume 
  $\Res_G^{G\times \Delta}[S\tensor_{\crp{F}}\crp{L}] 
   \in \SC_{\macl}(G,\crp{K})$,
  it follows~\cite[Corollary~9.2]{ladisch15b} 
  that 
  \[[S\tensor_{\crp{F}}\crp{L}]^{\abs{\Delta}}
   \in \SC_{\macl}(G\times\Delta,\crp{K}).
   \]
  Applying $C(\Delta)$ yields, by 
  Lemma~\ref{l:galoiscentrsc}, that
  $[S]^{\abs{\Delta}}\in\SC_{\macl}(G,\crp{E})$.
\end{proof}

Recall that we want to prove that an arbitrary
element $a\in \SC(G, \crp{E})$ is in fact
an element of $\SC_{\cyccl}(G,\crp{E})$,
and that we have already reduced to the case where
$a$ has $p$\nbd power order and
$G$ is a $p$\nbd group.
The results of this section yield a further reduction:
\begin{cor}\label{c:reducbigfield}
  Let $G$ be a $p$-group and 
  let 
  \[ a =\brcls{\theta}{\kappa\colon \oG\to G}{\crp{F}}
       \in \SC(G,\crp{E})
  \]
  have $p$\nbd power order.
  (So 
  $\theta$ is semi-invariant over $\crp{F}$ in $\oG$
  and $\crp{F}(\theta)=\crp{E}$.)  
  Let $\eps$ be a primitive $\abs{\oG}$-th root of unity.
  Then there is a unique field\/ $\crp{L}$ such that\/
  $\abs{\crp{L}:\crp{F}}$ is a $p'$\nbd number and
  $\abs{\crp{F}(\eps):\crp{L}}$ is a $p$\nbd power.
  If 
  $\brcls{\theta}{\kappa}{\crp{L}}
   \in \SC_{\cyccl}(G, \crp{L}(\theta))$,
  then 
  $a =\brcls{\theta}{\kappa}{\crp{F}}
   \in \SC_{\cyccl}(G,\crp{E})$.
\end{cor}
\begin{proof}
  The field $\crp{L}$ is uniquely determined as the fixed field
  of a Sylow $p$\nbd subgroup of the abelian Galois group
  $\Gal(\crp{F}(\eps)/\crp{F})$.
  
  By assumption, $\crp{E}=\crp{F}(\theta)$.
  By a result from my previous paper~\cite[Lemma~6.1]{ladisch15b},
  we know that 
  $\brcls{\theta}{\kappa}{\crp{F}}
    = \brcls{\theta}{\kappa}{\crp{E}^G}$.
  Therefore, we may assume without loss of generality that
  $\crp{E}^G=\crp{F}$.
  Then $\abs{\crp{E}:\crp{F}}$ is a power of $p$,
  since $G$ is a $p$\nbd group.
  Thus $\crp{E}\cap \crp{L}=\crp{F}$.
  With 
  $\crp{K}=\crp{E}\crp{L}\iso\crp{E}\tensor_{\crp{F}}\crp{L}$,
  the assumptions from the beginning of the section hold.
   
  Now 
  $\crp{L}(\theta)=\crp{L}\crp{F}(\theta)=\crp{L}\crp{E}=\crp{K}$
  and 
  $\brcls{\theta}{\kappa}{\crp{L}}
   = \brcls{\theta}{\kappa}{\crp{F}}\tensor_{\crp{F}}\crp{L}$.
  If 
  $\brcls{\theta}{\kappa}{\crp{L}}\in \SC_{\macl}(G,\crp{K})$
  then $\brcls{\theta}{\kappa}{\crp{F}}^{\abs{\crp{L}:\crp{F}}}\in 
       \SC_{\macl}(G,\crp{E})$
  by Lemma~\ref{l:fieldpowerred}.
  Since $a= \brcls{\theta}{\kappa}{\crp{F}}$ 
  has $p$-power order, it follows that
  $a\in \erz{a^{\abs{\crp{L}:\crp{F}}}}
   \subseteq \SC_{\macl}(G,\crp{E})=\SC_{\cyccl}(G,\crp{E})$.
  This proves the corollary.
\end{proof}

\section{Reduction to elementary groups}
Let $\crp{L}$ be a field of characteristic $0$.
Recall that a group $H$ is called
\defemph{$\crp{L}$-elementary for the prime $p$}
or $\crp{L}$-$p$-elementary, if the following two conditions hold:
\begin{enums}
\item $H=PC$ is the semidirect product of a normal cyclic
      $p'$\nbd group $C$ and a $p$\nbd group $P$.
\item The linear characters of $C$ are
      semi-invariant over $\crp{L}$ in $H$.
\end{enums}
The second condition is sometimes expressed differently.
We explain the connection.
Let $\lambda\in \Lin C$ be a faithful character of $C$.
Then $\lambda$ is semi-invariant in $H=PC$ if
for every $y\in P$ there is 
      $\sigma\in \Gal(\crp{L}(\lambda)/\crp{L})$
      such that $\lambda^{y\sigma}=\lambda$.
Now $\crp{L}(\lambda)=\crp{L}(\zeta)$,
where $\zeta$ is a primitive 
$\abs{C}$\nbd th root of unity.
Then for 
$\sigma\in \Gal(\crp{L}(\zeta)/\crp{L})$, 
there is a unique $k=k(\sigma)\in (\ints/\abs{C}\ints)^*$
with $\zeta^{\sigma}= \zeta^{k(\sigma)}$.
The second condition above is equivalent to:
For every $y\in P$ there is 
$\sigma\in \Gal(\crp{L}(\zeta)/\crp{L})$
such that $c^y = c^{k(\sigma)}$ for all $c\in C$.

We need the following part of the generalized induction 
theorem:
\begin{prop}[{Brauer, Berman, Witt~\cite[\S~12.6]{SerreLRFG}}]
  \label{p:bermanwittp}
  Let $\oG$ be a finite group with
  $\abs{\oG}=bp^k$, where the prime $p$
  does not divide $b$, and let $\crp{L}$
  be a field.
  Then we can write 
  \[ b1_{\oG} = \sum_{X,\tau} a_{X,\tau}\tau^{\oG}
  \]
  with integers $a_{X,\tau}$,
  where $(X,\tau)$ runs through the set of pairs
  such that $X$ is an $\crp{L}$-$p$-elementary subgroup of $\oG$
  and
  $\tau$ is the character of an $\crp{L}X$\nbd module.  
\end{prop}
The next result is similar to results of
Dade~\cites{dade70b}[cf.][Theorem~8.24]{isaCTdov}
and Schmid\cite[Lemma~5.1]{schmid88}.
See also~\cite[Proposition~3.3]{Yamada74}.
\begin{thm}\label{t:elemsubext}
  Let $\oG$ be a finite group with normal subgroup $N\nteq \oG$
  such that $\oG/N$
  is a $p$\nbd group for some prime $p$.
  Let $\theta\in \Irr (N)$ be 
  $\crp{L}$\nbd semi-invariant,
  where $\crp{L}$ is some field
  of characteristic $0$ such that 
  $\abs{\crp{L}(\zeta):\crp{L}}$ is a $p$\nbd power,
  where $\zeta$ is an $\exp(\oG)$-th root of unity.
  Then there exists an 
  $\crp{L}$-$p$-elementary group 
  $X\leq \oG$ and 
  $\phi\in \Irr(N\cap X)$ such that the following hold:
  \begin{enums}
  \item $\oG=XN$.
  \item $\ipch{\theta}{\phi^N}\not\equiv 0 \mod p$.
  \item \label{i:elemsubext_fields} $\crp{L}(\theta)=\crp{L}(\phi)$.
  \item \label{i:elemsubext_galois} For each $x\in X$ there is 
       $\alpha=\alpha_x\in \Gal(\crp{L}(\theta)/\crp{L})$
       such that 
       $\theta^{x \alpha}=\theta$ and
       $\phi^{x\alpha} = \phi$.
  \end{enums}  
\end{thm}
\begin{proof}
  Write
  \[ b1_{\oG} = \sum_{X,\tau} a_{X,\tau}\tau^{\oG},\]
  as in Proposition~\ref{p:bermanwittp}.
  Set $\Gamma= \Gal(\crp{L}(\theta)/\crp{L})$
  and 
  \[\psi = \Tr^{\crp{L}(\theta)}_{\crp{L}}(\theta)
      = \sum_{\sigma\in \Gamma} \theta^{\sigma}.
      \]
  Then
  \begin{align*}
    b\abs{\Gamma}
     &= \ipch{\psi}{\psi b1_{\oG}}_N\\
     &= \sum_{X,\tau} a_{X,\tau} 
        \ipch{\psi \cconj{\psi}}{(\tau^{\oG})_N}
        \\
     &= \sum_{X,\tau} a_{X,\tau}
        \ipch{\psi\cconj{\psi}}{\sum_{t\in [\oG:XN]}
                          \left(\tau^t_{X^t\cap N} \right)^N} 
     &&\quad\text{(Mackey)}                     
     \\
     &= \sum_{X,\tau} a_{X,\tau} \abs{\oG:XN}
                \ipch{\psi\cconj{\psi}}{(\tau_{X\cap N})^N}
     &&\quad\text{($\psi\cconj{\psi}$ is invariant in $\oG$)}
     \\
     &= \sum_{X,\tau} a_{X,\tau} \abs{\oG:XN}
             \ipch{\psi}{\psi(\tau_{X\cap N})^N}
     \\
     &= \abs{\Gamma} \sum_{X,\tau} a_{X,\tau} \abs{\oG:XN}
              \ipch{\theta}{\psi (\tau_{X\cap N})^N},            
  \end{align*}
  since $\psi (\tau_{X\cap N})^N$ is a character
  with values in $\crp{L}$.
  Since $b\not\equiv 0\mod p$, it follows that there is
  an $\crp{L}$\nbd $p$-elementary subgroup $X$ and
  a character $\tau$ of an $\crp{L}X$\nbd module, such that
  \[ a_{X,\tau}\abs{\oG:XN} \ipch{\theta}{\psi(\tau_{X\cap N})^N}
     \not\equiv 0\mod p.
  \]
  Fix such an $X$ and $\tau$ and set $Y=X\cap N$.
  First we note that since $\oG/N$ is a $p$\nbd group
  and $\abs{\oG:XN}\not\equiv 0\mod p$, 
  it follows that $\oG=XN$.
  
  Let $\widehat{\Gamma}=\Gal(\crp{L}(\zeta)/\crp{L})$
  and set $P= (X/Y)\times \widehat{\Gamma}$.
  This group acts  
  on the characters of $N$ and of $Y$. 
  The character $\tau_Y$ is invariant under the action of $P$,
  since it is the restriction from a character of $X$
  which has values in $\crp{L}$.
  Since $\theta$ is semi-invariant in $\oG$,
  it follows that every $P$-conjugate of $\theta$
  is of the form $\theta^{\sigma}$ with 
  $\sigma\in \Gamma=\Gal(\crp{L}(\theta)/\crp{L})$
  and thus $\psi$ is $P$\nbd invariant, as is $\psi_Y$, of course.
  Thus $\psi_Y\tau_Y$ is invariant under $P$.
  
  For each $\phi\in \Irr Y$, let
  $S(\phi)$ be the sum of the characters in the
  $P$\nbd orbit of $\phi$.
  Since $\psi_Y\tau_Y$ is invariant, we may write
  \[ \psi_Y\tau_Y
        = \sum_{\phi}c_{\phi} S(\phi),
        \qquad (c_{\phi} = \ipch{\psi_Y\tau_Y}{\phi}_Y)
  \]  
  where the sum runs over a set of representatives
  of the $P$\nbd orbits of $\Irr Y$.
  It follows from 
  \[ 0\not\equiv \ipch{\theta}{\psi\tau_Y^N}
      = \ipch{\theta_Y}{\psi_Y\tau_Y}
      = \sum_{\phi} c_{\phi}\ipch{\theta_Y}{S(\phi)}
  \]
  that there is $\phi\in \Irr Y$ such that
  $c_{\phi}\ipch{\theta_Y}{S(\phi)}\not\equiv 0 \mod p$.
  After replacing $\phi$ by another character in its
  orbit (if necessary), we may assume that 
  $\ipch{\theta_Y}{\phi}\not\equiv 0 \mod p$.
  For the rest of the proof, we
  fix a $\phi\in \Irr Y$ such that
  \[ \ipch{\psi_Y\tau_Y}{\phi}
     \ipch{\theta_Y}{S(\phi)} 
     \ipch{\theta_Y}{\phi}
     \not\equiv 0\mod p,
  \]
  and we show that this $\phi$ has the desired properties.  
  Of course we already have that 
  $\ipch{\theta}{\phi^N} = \ipch{\theta_Y}{\phi} 
   \not\equiv 0 \mod p$.
  It remains to show~\ref{i:elemsubext_fields} 
  and~\ref{i:elemsubext_galois}.
  
  The group $P=(X/Y)\times \widehat{\Gamma}$
  (with $\widehat{\Gamma} = \Gal(\crp{L}(\zeta)/\crp{L})$ as above)
  is a $p$\nbd group by our assumption.
  Write $P_{\phi}$ and $P_{\theta}$ for the stabilizers
  of $\phi$ and $\theta$ in $P$.
  We claim that $P_{\phi}=P_{\theta}$.
  First we show $P_{\theta}\subseteq P_{\phi}$.
  Since $\theta$ is semi-invariant in $X$,
  we have $P_{\theta}\widehat{\Gamma}=P$.
  Thus
  $P_{\theta}$ is normal in $P$ 
  and $P/P_{\theta}\iso X/X_{\theta}\iso\Gamma$ is abelian.
  Thus every $P_{\theta}$\nbd orbit contained in the $P$\nbd orbit
  of $\phi$ has the same length 
  $\abs{P_{\theta}:P_{\theta}\cap P_{\phi}}$.
  Since $\ipch{\theta_Y}{\phi^{xy}}= \ipch{\theta_Y}{\phi^x}$
  for $y\in P_{\theta}$, it follows that 
  $\abs{P_{\theta}:P_{\theta}\cap P_{\phi}}$ divides
  \[ 
    \ipch{\theta_Y}{S(\phi)} 
    = \sum_{x\in [P:P_{\phi}] } \ipch{\theta_Y}{\phi^x}
    \not\equiv 0\mod p
  .\]
  Thus $\abs{P_{\theta}:P_{\theta}\cap P_{\phi}}=1$ and 
  $P_{\theta}\subseteq P_{\phi}$.
  
  As we mentioned before, $\psi$ is the $P$\nbd orbit sum of $\theta$
  and $\tau_Y$ is $P$\nbd invariant.
  It follows that
  \[ \ipch{\psi_Y\tau_Y}{\phi}
         = \sum_{x\in [P:P_{\theta}]}\ipch{(\theta^x)_Y\tau_Y}{\phi}
         \not\equiv 0\mod p
  \]
  is divisible by the $p$\nbd power $\abs{P_{\phi}:P_{\theta}}$.
  Therefore, $P_{\phi} = P_{\theta}$ as claimed.
  
  From 
  $P_{\theta}\cap \widehat{\Gamma} 
   = P_{\phi}\cap\widehat{\Gamma}$
  it follows that a field automorphism over $\crp{L}$ fixes 
  $\phi$ if and only if it fixes $\theta$.
  Thus $\crp{L}(\theta)= \crp{L}(\phi)$. 
  Since $\theta$ is semi-invariant in $\oG$,
  there is, for every $x\in X$,
  an $\alpha_x\in \Gamma$ such that 
  $\theta^{x\alpha_x}=\theta$.
  If $\beta\in \widehat{\Gamma}$ is any extension
  of $\alpha_x$ to $\crp{L}(\zeta)$, then
  $(xY, \beta)\in P_{\theta}=P_{\phi}$.
  It follows that also
  $\phi^{x\alpha_x}=\phi$.
  The proof is finished. 
\end{proof}
\begin{cor}\label{c:elemsubext}
  In the situation of Theorem~\ref{t:elemsubext},
  let $\kappa\colon \oG \to G$ be an epimorphism
  with kernel $N$.
  Then 
  $\brcls{\theta}{\kappa}{\crp{L}}
   = \brcls{\phi}{\kappa_{|X}}{\crp{L}}$.   
\end{cor}
\begin{proof}
  It follows from Theorem~\ref{t:elemsubext} that
  Hypothesis~\ref{h:bconf2} holds for
  \[
    \begin{tikzcd}
      1 \rar & Y \rar \dar[hook] & X \rar{\kappa_{|X}} \dar[hook]
                                 & G \rar \dar[equals] & 1 \\
      1 \rar & N \rar & \oG \rar{\kappa} & G \rar & 1
    \end{tikzcd}
  \]
  over the field $\crp{L}$.  
  Thus Theorem~\ref{t:comparesubext} applies and yields that
  \[
    \brcls{\phi}{\kappa_{|X}}{\crp{L}}
     = [S]\cdot \brcls{\theta}{\kappa}{\crp{L}},
  \]
  where $[S]$ is the equivalence class of the algebra 
  $S= (i\crp{L}N i)^Y$, with
  \[ i = \sum_{\alpha\in \Gal(\crp{L}(\theta)/\crp{L})}
          (e_{\theta}e_{\phi})^{\alpha}.
  \]
  It remains to show that $[S]=1$, in other words,
  $S$ is a trivial $G$\nbd algebra over $\crp{L}(\theta)$.
  
  Pick $G$-algebras $A\in \brcls{\theta}{\kappa}{\crp{L}}$
  and $B\in \brcls{\phi}{\kappa_{|X}}{\crp{L}}$
  with $A\tensor_{\crp{L}(\phi)}S \iso B$, as in the proof of
  Theorem~\ref{t:comparesubext}.
  By Lemma~\ref{l:StensorFLfiFKi}, we know that
  $S$ is central simple of dimension $n^2$ over $\crp{L}(\theta)$,
  where $n=\ipch{\theta_Y}{\phi}\not\equiv 0\mod p$.
  Since $\crp{L}(\zeta)$ is a splitting field of all groups involved,
  it follows that 
  $A\tensor_{\crp{L}}\crp{L}(\zeta)$ is a direct product 
  of matrix rings over $\crp{L}(\zeta)$,
  and $A\tensor_{\crp{L}(\theta)}\crp{L}(\zeta)$ is a matrix ring
  over $\crp{L}(\zeta)$.
  The same holds for $B$, and thus 
  $S\tensor_{\crp{L}(\theta)}\crp{L}(\zeta)\iso\mat_n(\crp{L}(\zeta))$.
  Since $\abs{\crp{L}(\zeta):\crp{L}(\theta)}$ is a power of $p$ and
  $n$ is prime to $p$, it follows that 
  $S\iso \mat_n(\crp{L}(\theta))$.
  
  The action of $G$ on $\crp{L}(\theta)$ extends naturally 
  to an action
  of $G$ on $\mat_n(\crp{L}(\theta))$.
  Via some fixed isomorphism $S\iso \mat_n(\crp{L}(\theta))$,
  we get an action of $G$ on $S$.
  We write $\eps\colon G\to \Aut S$ for this action,
  that is we write $s^{\eps(g)}$ to denote the effect of this action.
  This action is different from the $G$-algebra action on
  $S=(i\crp{L}Ni)^Y$ we already have, 
  but both actions agree on $\Z(S)\iso \crp{L}(\theta)$.
  
  It follows that for $g\in G$, the map
  $S\ni s \mapsto s^{\eps(g^{-1})g}$
  is an $\crp{L}(\theta)$\nbd algebra automorphism of $S$.
  Thus there is $\sigma(g)\in S^*$ such that
  $s^{\eps(g^{-1})g}= \sigma(g)^{-1}s \sigma(g)$ for all $s\in S$,
  or $s^{g}= s^{\eps(g)\sigma(g)}$.  
  By comparing left and right hand side of
  \[ s^{\eps(xy)\sigma(xy) }
      = s^{xy} 
      = s^{\eps(x)\sigma(x)\eps(y)\sigma(y)}
      = s^{\eps(xy)\sigma(x)^{\eps(y)}\sigma(y)},
  \]
  we see that
  \[ \sigma(x)^{\eps(y)}\sigma(y) = \alpha(x,y)\sigma(xy)
       \text{ for some } \alpha(x,y)\in \crp{L}(\theta)^{*}.\]
  Thus 
  $\sigma\colon G\to S\iso\mat_n(\crp{L}(\theta))$
  is a twisted projective representation and
  $\alpha \in Z^2(G, \crp{L}(\theta)^{*})$.
  
  Taking determinants in 
  $\sigma(x)^{\eps(y)}\sigma(y) = \alpha(x,y)\sigma(xy)$
  yields that
  $\alpha^n$ is a coboundary.
  We know also that $\alpha^{\abs{G}}$ is a coboundary.
  Since $\abs{G}$ is a power of $p$ and $n$ is prime to $p$,
  it follows that $\alpha$ itself is a coboundary.
  So after multiplying $\sigma$ with a coboundary, 
  we may assume that 
  $\sigma(x)^{\eps(y)}\sigma(y)=\sigma(xy)$ for all
  $x$, $y\in G$.
  
  Now we can show that $S$ is a trivial $G$\nbd algebra.
  Let $V=\crp{L}(\theta)^n$ and view $V$ as a right $S$\nbd module
  via our fixed isomorphism $S\iso\mat_n(\crp{L}(\theta))$.
  We have to define a right $G$\nbd module structure on $V$
  such that $S\iso\enmo_{\crp{L}(\theta)}(V)$
  as $G$\nbd algebras over $\crp{L}(\theta)$.
  
  The action of $G$ on $\crp{L}(\theta)$ defines an action of
  $G$ on $V$ which we denote by $v\mapsto v^g$.
  It has the property
  \[ (vs)^g = v^g s^{\eps(g)}
     \quad \text{for $s\in S$}.
  \]
  Define $v\circ g := v^g \sigma(g)$.
  Then 
  \begin{align*}
    (v\circ x)\circ y
    &= (v^x \sigma(x))^y \sigma(y)
     = v^{xy} \sigma(x)^{\eps(y)} \sigma(y)
     = v^{xy} \sigma(xy)
     = v\circ (xy)
  \end{align*}
  for $v\in V$ and $x$, $y\in G$, and
  $(v\lambda)\circ g = (v\circ g)\lambda^g$ for
  $v\in V$, $\lambda\in \crp{L}(\theta)$ and $g\in G$.
  Thus $V$ is a right module over the crossed product 
  $\crp{L}(\theta)G$.
  The computation
  \begin{align*} 
   \left( (v\circ g^{-1}) s \right) \circ g
        = \left( v^{g^{-1}} \sigma(g^{-1}) s \right)^g \sigma(g) 
        &= v \sigma(g^{-1})^{\eps(g)} s^{\eps(g)} \sigma(g)
        \\
        &= v \sigma(g)^{-1} s^{\eps(g)} \sigma(g)
        \\
        &= v s^g
  \end{align*}
  shows that
  $S\iso \enmo_{\crp{L}(\theta)}(V)$
  as $G$\nbd algebras.
  This finishes the proof that $[S]=1$ in
  $\BrCliff(G,\crp{L}(\theta))$.
\end{proof}
\section{Proof of the theorem for elementary groups}
Let $\crp{F}$ be a field of characteristic $0$.
We need the following easy lemma which states that
semi-invariance is in some sense transitive:
\begin{lemma}\label{l:semi-inv-trans}
  Let $M$, $N$ be normal subgroups of $\oG$ with
  $M\leq N$.
  Suppose that $\tau\in \Irr M$
  is $\crp{F}$\nbd semi-invariant in $N$ and that
  $\theta\in \Irr(N\mid \tau)$ is $\crp{F}$\nbd semi-invariant
  in $\oG$.
  Then $\tau$ is $\crp{F}$\nbd semi-invariant in $\oG$.
\end{lemma}
\begin{proof}
  Let $g\in \oG$.
  There is $\alpha\in \Gal(\crp{F}(\theta)/\crp{F})$
  such that $\theta^{g\alpha}=\theta$.
  We may extend $\alpha$ to $\crp{F}(\theta,\tau)$.
  For simplicity, we denote this element of
  $\Gal(\crp{F}(\theta,\tau)/\crp{F})$ by $\alpha$, too.
  Then $\tau^{g\alpha}$ is a constituent of
  $\theta^{g\alpha}_M=\theta_M$, and so there is
  $n\in N$ such that 
  $\tau^{g\alpha}=\tau^n$.
  Since $\tau $ is semi-invariant in $N$, there
  is $\beta\in \Gal(\crp{F}(\tau)/\crp{F})$
  such that $\tau^{n\beta}=\tau$.
  Thus $\tau^{g\alpha\beta}=\tau$.
  Since $g\in \oG$ was arbitrary,
  $\tau$ is semi-invariant in $\oG$ over $\crp{F}$.
\end{proof}
Let $\oG$ be an $\crp{F}$\nbd elementary group
with respect to some prime $p$.
Then $\oG=PC$ where $P$ is a Sylow $p$\nbd subgroup
and $C$ is a normal cyclic subgroup of $p'$\nbd order.
In the next result, we will only need that
$\oG$ has a normal abelian subgroup $C$ such that
$\oG/C$ is nilpotent.
In this case, every irreducible character of
$\oG$ is induced from a linear character
of some subgroup containing the 
normal subgroup $C$~\cite[Theorem~6.22]{isaCTdov}.
(This is even true if $\oG/C$ is supersolvable,
but we will also need the nilpotence of $\oG/C$.)
\begin{thm}\label{t:reduclin}
  Let
  \[\begin{tikzcd}
                  1 \rar & N \rar & \oG \rar{\kappa} & G \rar & 1
     \end{tikzcd}  
  \]
  be an exact sequence of groups.
  Assume that there is 
  a normal abelian subgroup $C\leq N$ such that
  $\oG/C$ is nilpotent.
  Suppose that 
  $\theta\in \Irr N$ is 
  $\crp{F}$\nbd semi-invariant in $\oG$.
  Then there are subgroups
  $ A\leq M\leq H$ of $\oG$ and a linear
  character $\lambda\in \Lin A$ such that 
  \begin{enums}[series=subextprops]
  \item \label{i:reduclingroups}
        $A\nteq H$, $\oG=HN$ and $C\leq A\leq M=H\cap N$,
  \item \label{i:reduclinind}
         $\theta= \lambda^N$,
  \item \label{i:fieldequ}
        $\crp{F}(\theta)=\crp{F}(\phi)$,
         where $\phi=\lambda^M$,
  \item \label{i:brcliffequ}
        $\brcls{\theta}{\kappa}{\crp{F}}
         = \brcls{\phi}{\kappa_{|H}}{\crp{F}}$.
  \end{enums}  
\end{thm}
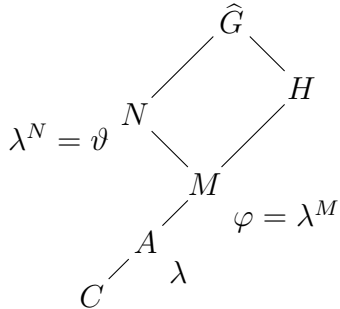
\begin{center}
  \begin{tikzpicture}[on grid=true, inner sep=2pt]
                  \node (G) {$\oG$};
                  \node (N) [below left=1.8 of G,
                             label=190:{$\lambda^N=\theta$}] 
                            {$N$};
                  \node (H) [below right=1.3 of G] 
                            {$H$};
                  \node (M) [below right=1.3 of N,
                             label={-25:{$\phi=\lambda^M$}} ] 
                            {$M$};  
                  \node (A) [below left=1.1 of M,
                             label={-35:$\lambda$}] 
                            {$A$};
                  \node (C) [below left= 1 of A] 
                            {$C$};           
                  \draw (C)--(A)--(M)--(H)--(G)--(N)--(M); 
  \end{tikzpicture}
\end{center}
Our proof of this result follows closely the proof
of Proposition~3.6 in Yamada's book~\cite{Yamada74}.
(Proposition~3.6 in Yamada's book is the case $N=\oG$ (that is, $G=1$)
of Theorem~\ref{t:reduclin},\ref{i:reduclingroups}--\ref{i:fieldequ}.)
\begin{proof}[Proof of Theorem~\ref{t:reduclin}]
  Consider the set of all subgroups
  $T\subseteq N$ such that
  \begin{enumerate}[label=(\arabic*)]
  \item \label{i:auxsubgrnormal}
        $C\leq T\nteq \oG$,
  \item \label{i:auxsubgrind}
        there is $\tau\in \Irr T$ such that
        $\theta=\tau^N$,
  \item \label{i:auxsubgrsemiinv}
        for all $n\in N$, there is
        $\alpha\in \Gal(\crp{F}(\tau)/\crp{F})$
        such that $\tau^{n\alpha}=\tau$.
  \end{enumerate}  
  Let $T$ be minimal among the subgroups
  having Properties~\ref{i:auxsubgrnormal}--\ref{i:auxsubgrsemiinv}.
  (Note that $T=N$ is such a subgroup.)
  If the corresponding 
  $\tau$ is linear, we set
  $A=T$, $\lambda=\tau$, $H=\oG$ and
  $M=N$. Then $\phi=\lambda^M=\theta$
  and~\ref{i:reduclingroups}--\ref{i:brcliffequ} 
  of the proposition are trivially true.
  So in the case where $\tau(1)=1$, the proof is finished.
  
  Assume $\tau(1)>1$.
  Our first goal is to find a proper subgroup $T_0 < T$
  which is normal in $\oG$ and such that
  $\tau$ is induced from a character $\tau_0\in \Irr T_0$.
  The non-linear character
  $\tau$ is induced from a linear character
  of some subgroup $B$ such that $C\leq B < T$.
  Let $S$ be 
  a maximal subgroup of $T$  containing $B$, and thus $C$. 
  Then $\tau$ is induced from some character of $S$.
  Since $S\nt T$ (because $T/C$ is nilpotent), we see that
  $\tau_{T\setminus S}=0$.
  Set 
  \[ U = \bigcap_{g\in \oG}S^g.
  \]
  Then $C\leq U\nteq \oG$.
  We claim that $\tau$ vanishes on $T\setminus U$.
  Since $\tau$ vanishes outside $S$,
  it follows that $\tau^g$ vanishes outside $S^g$ for $g\in \oG$.
  But $\tau^g$ and $\tau$ are Galois conjugate
  (by Lemma~\ref{l:semi-inv-trans}), and thus
  $\tau $ vanishes outside $S^g$.
  Since $g\in \oG$ was arbitrary, it follows
  that $\tau $ vanishes outside $U$ as claimed.
  
  Since $ U$, $T\nteq \oG$ and $\oG/U$ is nilpotent, 
  there exists 
  $T_0\nteq \oG$ such that
  $U\leq T_0 < T$ and
  $\abs{T:T_0}$ is a prime $p$.
  Since $\tau$ vanishes outside $U$, it vanishes
  outside $T_0$.
  Since $\abs{T:T_0}=p$, we see that
  $\tau=\tau_0^T$ for some
  $\tau_0\in \Irr T_0$.
  We have found a $T_0$ as wanted.
  
  In the following, set
  \[ \oG_0 = \{g\in \oG \mid 
                   \exists \alpha\in \Gal(\crp{F}(\tau_0)/\crp{F})
                   \colon \tau_0^{g\alpha}=\tau_0
             \},
  \]
  $N_0=\oG_0\cap N$ and $\theta_0 = \tau_0^{N_0}$.
  Then 
  $\theta_0^N = \tau_0^N = (\tau_0^T)^N = \tau^N=\theta$
  and $\theta_0\in \Irr (N_0\mid \tau_0)$ is the unique 
  element with $\theta_0^N=\theta$ (Lemma~\ref{l:cliffcorfields}).
  Corollary~\ref{c:subextclifford}, applied to 
  the extension
  \[
    \begin{tikzcd}
      1 \rar & N \rar & \oG \rar{\kappa} & G \rar & 1
    \end{tikzcd}
  \] 
  and the normal subgroup $T_0\subseteq N$,
  yields that the following hold:
  \begin{enums}
  \item $\oG= \oG_0 N$,
  \item $\theta=\theta_0^N$,
  \item $\crp{F}(\theta)= \crp{F}(\theta_0)$,
  \item $\brcls{\theta}{\kappa}{\crp{F}}
         = \brcls{\theta_0}{\kappa_{|\oG_0}}{\crp{F}}$.
  \end{enums}
  (We can also show that $\oG=\oG_0 T$ by applying
   Corollary~\ref{c:subextclifford}
   to the extension 
   $ T \to \oG \to \oG/T$,
   but actually we will not need this for the proof.)
  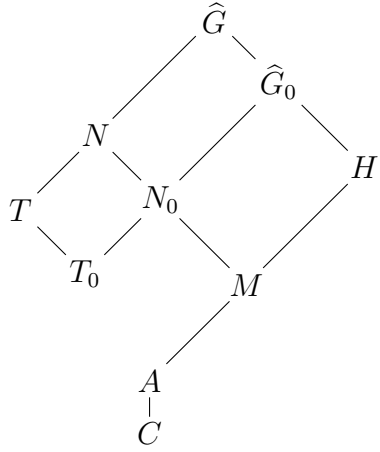
\begin{figure}[ht]
    \begin{tikzpicture}[on grid=true, inner sep=2pt]
         \node (G) {$\oG$};
         \node (G0) [below right=1.2 of G] {$\oG_0$};
         \node (H) [below right=1.6 of G0] {$H$};
         \node (N) [below left=2.2 of G, ] {$N$};
         \node (N0) [below left=2.2 of G0,] {$N_0$};
         \node (M) [below left=2.2 of H, ]{$M$};  
         \node (T) [below left=1.4 of N, ] {$T$};
         \node (T0) [below left=1.4 of N0, ] {$T_0$};
         \node (A) [below left=1.8 of M, ] {$A$};
         \node (C) [below= 0.7 of A] {$C$};
         \draw (C)--(A)--(M)--(H)--(G0)--(G)--(N)--(N0)--(M);
         \draw (N)--(T)--(T0)--(N0)--(G0);
    \end{tikzpicture}
    \caption{Proof of Theorem~\ref{t:reduclin}}
    \label{fig:reduclinproof}
  \end{figure}

  Note that~\ref{i:auxsubgrnormal} and~\ref{i:auxsubgrind}
  hold for $T_0$ and $\tau_0$.     
  It follows from the minimality of $T$ that~\ref{i:auxsubgrsemiinv}
  does not hold for $T_0$, which means that 
  $N_0 < N$.   
  Thus we may apply induction (on $\abs{N/C}$)
  to conclude that there are subgroups
  $A\leq M\leq H\leq \oG_0$ and 
  $\lambda\in \Lin A$ 
  such that Properties~\ref{i:reduclingroups}--\ref{i:brcliffequ}
  hold with $\oG$, $N$, $\theta$
  replaced by $\oG_0$, $N_0$, $\theta_0$.
  It follows that
  $\oG= \oG_0 N= HN$ and $M=H\cap N_0 = H\cap \oG_0\cap N=H\cap N$
  (see Figure~\ref{fig:reduclinproof}),
  that $\theta=\theta_0^N= \lambda^N$,
  and that $\crp{F}(\theta)=\crp{F}(\theta_0)=\crp{F}(\phi)$.
  Finally, we have that
  \[\brcls{\theta}{\kappa}{\crp{F}}
    = \brcls{\theta_0}{\kappa_{|\oG_0}}{\crp{F}}
    = \brcls{\phi}{\kappa_{|H}}{\crp{F}}.
    \qedhere
  \]
\end{proof}
\begin{cor}\label{c:reduccyc}
  In the situation of Theorem~\ref{t:reduclin},
  $\brcls{\theta}{\kappa}{\crp{F}}
   \in \SC_{\cyccl}(G,\crp{F}(\theta))$.
\end{cor}
\begin{proof}
  By Proposition~\ref{p:cycltoabel}, it suffices to show that 
  $\brcls{\theta}{\kappa}{\crp{F}}
     \in \SC_{\macl}(G,\crp{F}(\theta))$.
  By Theorem~\ref{t:reduclin},
  we have 
  $\brcls{\theta}{\kappa}{\crp{F}}
           = \brcls{\phi}{\kappa_{|H}}{\crp{F}}$,
  and it is obvious that 
  the conditions in Theorem~\ref{t:reduclin}
  yield that
  $\brcls{\phi}{\kappa_{|H}}{\crp{F}}$
  is in $\SC_{\macl}(G,\crp{F}(\theta))$.
  (See Definition~\ref{defi:subclasses}.)           
\end{proof}
\begin{remark}
In fact, the construction from the proof
of Theorem~\ref{t:reduclin} itself yields
an $M$ that also has the following properties:
\begin{enums}[resume=subextprops]
  \item \label{i:galstab}
        $M=\{n\in N \mid 
        \exists \alpha\in \Gal(\crp{F}(\lambda)/\crp{F})
         \colon \lambda^{n\alpha}=\lambda\}$,
  \item \label{i:galiso}
        $M/A\iso \Gal(\crp{F}(\lambda)/\crp{F}(\theta))$.
\end{enums}
\end{remark}
\begin{proof} 
  By Lemma~\ref{l:galisoredundant},
  $\ref{i:galstab}$ together with $\lambda^M\in \Irr M$
  implies \ref{i:galiso}.
  
  To see~\ref{i:galstab}, 
  let $N_0$ be as in the proof above and write 
     $\widehat{M}= \{n\in N \mid 
             \exists \alpha\in \Gal(\crp{F}(\lambda)/\crp{F}(\theta))
              \colon \lambda^{n\alpha}=\lambda\}$.
     By induction we may assume that $\widehat{M}\cap N_0=M$.
     Since $\lambda^N=\theta\in \Irr N$
     and $\lambda^{N_0}=\theta_0\in \Irr N_0$,
     we know from \ref{i:galiso}
     respective Lemma~\ref{l:galisoredundant} that
     $\widehat{M}/A\iso \Gal(\crp{F}(\lambda)/\crp{F}(\theta))$
     and $M/A\iso \Gal(\crp{F}(\lambda)/\crp{F}(\theta_0))$.
     But since $\crp{F}(\theta)=\crp{F}(\theta_0)$,
     it follows that $M=\widehat{M}$.
     The proof is finished.   
\end{proof}

\section{Proof of the main theorem}
For the convenience of the reader, we summarize here how
Theorem~\ref{main} from the introduction follows from
the various results proved so far.
We restate Theorem~\ref{main}, using the notation introduced
in Section~\ref{sec:cycltoabel}.
\begin{thm}
  Let $\crp{E}$ be a field on which the finite group $G$ acts.
  Then $\SC(G,\crp{E})=\SC_{\cyccl}(G,\crp{E})$.
\end{thm}
\begin{proof}
  Let $a=\brcls{\theta}{\kappa}{\crp{F}}\in \SC(G,\crp{E})$.
  We have to show that $a\in \SC_{\cyccl}(G,\crp{E})$.
  By Lemmas~\ref{l:reducppart} and~\ref{l:reducpgroup},
  we may assume that 
  $a$ has $p$\nbd power order and that $G$ is a $p$\nbd group.
  By Corollary~\ref{c:reducbigfield}, it suffices to show that
  $\brcls{\theta}{\kappa}{\crp{L}}\in \SC_{\cyccl}(G, \crp{L}(\theta))$,
  where $\crp{L}$ is such that a splitting field of $\oG$ 
  and all its subgroups has $p$\nbd power degree over $\crp{L}$.
  Then Corollary~\ref{c:elemsubext} yields that
  we may assume that $\oG$ is $\crp{L}$\nbd elementary for the prime $p$.
  In this case, Corollary~\ref{c:reduccyc} yields that
  $\brcls{\theta}{\kappa}{\crp{L}}\in \SC_{\cyccl}(G,\crp{L}(\theta))$.  
\end{proof}
%
\printbibliography   
%
\end{document}